\documentclass[11pt]{article}
\usepackage[margin=1in]{geometry} 
\usepackage{amssymb,amsfonts,amsmath,bbm,mathrsfs,stmaryrd}
\usepackage{xcolor}
\usepackage{url}
 \usepackage{relsize}

\usepackage[colorlinks,
             linkcolor=black!75!red,
             citecolor=blue,
             pdftitle={BU},
             pdfauthor={Alexandru Chirvasitu, Benjamin Passer},
             pdfproducer={pdfLaTeX},
             pdfpagemode=None,
             bookmarksopen=true
             bookmarksnumbered=true]{hyperref}

\usepackage{tikz}
\usetikzlibrary{arrows,calc,decorations.pathreplacing,decorations.markings,intersections,shapes.geometric,through,fit,shapes.symbols,positioning,decorations.pathmorphing}

\usepackage[amsmath,thmmarks,hyperref]{ntheorem}
\usepackage{cleveref}

\numberwithin{equation}{section}

\crefname{section}{Section}{Sections}
\crefformat{section}{#2Section~#1#3} 
\Crefformat{section}{#2Section~#1#3} 

\crefname{subsection}{\S}{\S\S}
\crefformat{subsection}{#2\S#1#3} 
\Crefformat{subsection}{#2\S#1#3} 

\theoremstyle{plain}

\newtheorem{lemma}{Lemma}[section]
\newtheorem{proposition}[lemma]{Proposition}
\newtheorem{corollary}[lemma]{Corollary}
\newtheorem{theorem}[lemma]{Theorem}
\newtheorem{conjecture}[lemma]{Conjecture}
\newtheorem{question}[lemma]{Question}

\theoremstyle{nonumberplain}

\theoremstyle{plain}
\theorembodyfont{\upshape}
\theoremsymbol{\ensuremath{\blacklozenge}}

\newtheorem{definition}[lemma]{Definition}
\newtheorem{example}[lemma]{Example}
\newtheorem{remark}[lemma]{Remark}

\crefname{definition}{definition}{definitions}
\crefformat{definition}{#2definition~#1#3} 
\Crefformat{definition}{#2Definition~#1#3} 

\crefname{ex}{example}{examples}
\crefformat{example}{#2example~#1#3} 
\Crefformat{example}{#2Example~#1#3} 

\crefname{remark}{remark}{remarks}
\crefformat{remark}{#2remark~#1#3} 
\Crefformat{remark}{#2Remark~#1#3} 

\crefname{convention}{convention}{conventions}
\crefformat{convention}{#2convention~#1#3} 
\Crefformat{convention}{#2Convention~#1#3} 

\crefname{claim}{claim}{claims}
\crefformat{claim}{#2claim~#1#3} 
\Crefformat{claim}{#2Claim~#1#3} 

\crefname{conjecture}{conjecture}{conjectures}
\crefformat{conjecture}{#2conjecture~#1#3} 
\Crefformat{conjecture}{#2Conjecture~#1#3}

\crefname{lemma}{lemma}{lemmas}
\crefformat{lemma}{#2lemma~#1#3} 
\Crefformat{lemma}{#2Lemma~#1#3} 

\crefname{proposition}{proposition}{propositions}
\crefformat{proposition}{#2proposition~#1#3} 
\Crefformat{proposition}{#2Proposition~#1#3} 

\crefname{question}{question}{questions}
\crefformat{question}{#2question~#1#3} 
\Crefformat{question}{#2Question~#1#3}

\crefname{corollary}{corollary}{corollaries}
\crefformat{corollary}{#2corollary~#1#3} 
\Crefformat{corollary}{#2Corollary~#1#3} 

\crefname{theorem}{theorem}{theorems}
\crefformat{theorem}{#2theorem~#1#3} 
\Crefformat{theorem}{#2Theorem~#1#3} 

\crefname{assumption}{assumption}{Assumptions}
\crefformat{assumption}{#2assumption~#1#3} 
\Crefformat{assumption}{#2Assumption~#1#3} 

\crefname{equation}{}{}
\crefformat{equation}{(#2#1#3)} 
\Crefformat{equation}{(#2#1#3)}

\theoremstyle{nonumberplain}
\theoremsymbol{\ensuremath{\blacksquare}}

\newtheorem{proof}{Proof}

\newcommand\bC{{\mathbb C}}

\newcommand\bQ{{\mathbb Q}}
\newcommand\bR{{\mathbb R}}
\newcommand\bS{{\mathbb S}}
\newcommand\bT{{\mathbb T}}
\newcommand\bZ{{\mathbb Z}}

\newcommand\cH{{\mathcal H}}

\DeclareMathOperator{\id}{id}

\def\polhk#1{\setbox0=\hbox{#1}{\ooalign{\hidewidth
    \lower1.5ex\hbox{`}\hidewidth\crcr\unhbox0}}}
    
\newcommand{\trivdim}[2]{\mathrm{dim}^{#2}_{\mathrm{triv}}(#1)}
\newcommand{\join}[3]{#1 \circledast^{#2} #3}
\newcommand{\qu}[1]{\lq\lq #1\rq\rq\hspace{0pt}}
\newcommand{\bes}{\begin{equation*}}
\newcommand{\ees}{\end{equation*}}
\newcommand{\be}{\begin{equation}}
\newcommand{\ee}{\end{equation}}

\begin{document}

\title{Invariants in Noncommutative Dynamics}
\author{Alexandru Chirvasitu\footnote{University at Buffalo, \url{achirvas@buffalo.edu} \hspace{2 pt} Partial support through NSF grants DMS-1565226 and DMS-1801011.}, Benjamin Passer\footnote{University of Waterloo, \url{bpasser@uwaterloo.ca} \hspace{2 pt} Partial support from a Zuckerman Fellowship at Technion-Israel Institute of Technology.}}

\maketitle

\begin{abstract}
When a compact quantum group $H$ coacts freely on unital $C^*$-algebras $A$ and $B$, the existence of equivariant maps $A \to B$ may often be ruled out due to the incompatibility of some invariant. We examine the limitations of using invariants, both concretely and abstractly, to resolve the noncommutative Borsuk-Ulam conjectures of Baum-D\polhk{a}browski-Hajac. Among our results, we find that for certain finite-dimensional $H$, there can be no well-behaved invariant which solves the Type 1 conjecture for all free coactions of $H$. This claim is in stark contrast to the case when $H$ is finite-dimensional and abelian. In the same vein, it is possible for all iterated joins of $H$ to be cleft as comodules over the Hopf algebra associated to $H$. Finally, two commonly used invariants, the local-triviality dimension and the spectral count, may both change in a $\theta$-deformation procedure.
\end{abstract}

\setlength{\parskip}{.15 cm}


\section{Introduction}

The Borsuk-Ulam theorem, a landmark result concerning the antipodal action of the two-element group $\bZ/2\bZ$ on spheres $\bS^n$, severely restricts the existence of equivariant maps. Because the antipodal action on $\bS^n$ is generated by the order two homeomorphism $x \mapsto -x$, equivariant maps $f: \bS^n \to \bS^m$ by definition satisfy $f(-x) = -f(x)$, so they are also commonly called \textit{odd} functions.

\begin{theorem}[Borsuk-Ulam]
There is no odd function $f: \bS^n \to \bS^m$ if $m < n$.
\end{theorem}

Viewing the Borsuk-Ulam theorem through the Gelfand-Naimark correspondence $X \leftrightarrow C(X)$ between compact Hausdorff spaces and commutative unital $C^*$-algebras produces an equivalent $C^*$-algebraic statement. If $\bZ/2\bZ$ acts on $C(\bS^n)$ via $f(x) \mapsto f(-x)$, then there is no equivariant unital $*$-homomorphism $\phi: C(\bS^m) \to C(\bS^n)$ if $m < n$. Note in particular that the roles of domain and codomain have switched, as the Gelfand-Naimark correspondence is contravariant.

Just as the Borsuk-Ulam theorem may be generalized to larger classes of (para)compact spaces and groups acting on these spaces \cite{dold, volovikovindex}, so too does the $C^*$-algebraic formulation open itself up to broader interpretation \cite{BDH, dabsusp, bentheta, taghavi, qdeformed}.  The perspective taken in \cite{BDH} is that the sphere $\bS^{n}$ may be written as the topological join $(\bZ/2\bZ)^{*n+1}$ of $n+1$ copies of the group $\bZ/2\bZ$, where the antipodal action $x \mapsto -x$ is just the diagonal action of $\bZ/2\bZ$ on its iterated joins. Thus the Borsuk-Ulam theorem may be written as
\bes
\nexists \text{ continuous equivariant map } \hspace{.1 in} (\bZ/2\bZ)^{*n+1} \to (\bZ/2\bZ)^{*n}.
\ees
The above was immediately generalized to many more spaces and groups, as in the following topological Borsuk-Ulam conjecture.

\begin{conjecture}[Baum-D\polhk{a}browski-Hajac, \cite{BDH}]\label{conj:top}
Let $X$ be a compact Hausdorff space with a free action of a nontrivial compact group $G$, and equip the join $X * G$ with the diagonal action. Then there is no equivariant map $X * G \to X$.
\end{conjecture}

All objects in \Cref{conj:top} can be generalized into the $C^*$-algebraic setting. An advantage of this perspective is that one may consider coactions of compact \textit{quantum} groups on $C^*$-algebras which are not necessarily commutative. If $A$ is a unital $C^*$-algebra which admits a coaction $\delta$ of a compact quantum group $(H, \Delta)$, then the noncommutative join procedure of \cite{BDH, joinfusion} produces a join $\join{A}{\delta}{H}$ and an induced coaction $\delta_\Delta$ of $(H, \Delta)$. Moreover, if $\delta$ is free, then $\delta_\Delta$ is also free. Thus, Conjecture \ref{conj:top} may be reformulated in more general terms.

\begin{conjecture}[Baum-D\polhk{a}browski-Hajac, \cite{BDH}]\label{conj:type12}
Let $A$ be a unital $C^*$-algebra with a free coaction $\delta$ of a nontrivial compact quantum group $(H, \Delta)$. Equip the join $\join{A}{\delta}{H}$ with the induced free coaction $\delta_\Delta$. Then the following hold.
\begin{itemize}
\item Type 1 Conjecture: There is no $(\delta, \delta_\Delta)$-equivariant unital $*$-homomorphism $A \to \join{A}{\delta}{H}$.
\item Type 2 Conjecture: There is no $(\Delta, \delta_\Delta)$-equivariant unital $*$-homomorphism $H \to \join{A}{\delta}{H}$.
\end{itemize}
\end{conjecture}

The above conjectures are quite general. For example, if $\Gamma$ is a finite group, then a grading of $A$ by $\Gamma$ corresponds to a coaction of the compact quantum group $\bC \Gamma$ on $A$. Such gradings were considered in \cite{taghavi}, with some highlights on the difference between the abelian and non-abelian cases. Motivated by the construction in \cite[Example 3]{taghavi}, which connects an invertible element of a nontrivial isotypic subspace to $1$ using its representation as a commutator, we find that perfect groups $\Gamma$ are the source of many pathological constructions in the equivariant join setting.

If $H = C(G)$ is classical, then \Cref{conj:top} is equivalent to \Cref{conj:type12} Type 1 by \cite[Proposition 3.5]{alexbenjoin}. That is, for compact groups $G$, it makes no difference if one discusses noncommutative or commutative $C^*$-algebras on which $G$ acts freely. Moreover, there are numerous special cases (classical and quantum) for which the Type 1 and Type 2 conjectures have been resolved \cite{BDH, BDHrevisited, alexbenjoin, hajacindex, benfree}. The Type 1 conjecture remains open in full generality, and the Type 2 conjecture is false by \cite[Theorems 2.3 and 2.6]{alexbenjoin}, even if $H$ is commutative and $A$ is nuclear. However, the classification of precisely which $H$ produce Type 2 counterexamples is incomplete.

We begin by giving some background and notational conventions in \Cref{se.notation}. Next, in \Cref{se.type1}, we consider the invariants used to prove special cases of the Type 1 conjecture in \cite{benfree, hajacindex} and describe their limitations through the construction of extremal examples. In the classical case $H = C(G)$, $G$ a compact torsion-free group, we produce examples which show that actions of finite local-triviality dimension (which satisfy the Type 1 conjecture by \cite[Theorem 5.3]{hajacindex}) include cases which are not distinguished by previous invariants. However, constructions using perfect groups show that the local-triviality dimension, or rather, any well-behaved invariant, cannot universally solve the Type 1 conjecture, even if $H$ is finite-dimensional. Namely, there exists a finite compact quantum group $H$ and a free coaction of $H$ on $A$, such that $A$ admits maps into any iterated join of $H$. This claim is in stark contrast to the finite abelian setting, and our constructions give examples of $C^*$-algebras, both commutative and noncommutative, which should be key examples in later study.

In \Cref{se.type2}, we consider the Type 2 conjecture and revisit the classification of which $H$ admit counterexamples. In the classical abelian case $H = C^*(\Gamma)$, $\Gamma$ discrete and abelian, we show that the Type 2 conjecture fails if direct sums of the map induced by $\gamma \mapsto \gamma \oplus \gamma^{-1}$ contract to the trivial representation in a stabilized fashion. However, any reduced group $C^*$-algebra which satisfies the Baum-Connes conjecture and has nontrivial rational homology enjoys a non-contractibility property (and in particular, the reduced $C^*$-algebras considered need not admit characters).  Next, we study the characters of joins and fusions, concluding that they may be easily described in terms of the characters of the components. This allows us to write a connectedness property for equivariant joins when the compact quantum group $H$ has a counit. We conclude with discussion of cleftness when $H$ is finite-dimensional, showing that for certain $H$, all of the iterated joins of $H$ are cleft as comodules over the associated Hopf algebra. However, there might not exist cleaving maps which respect the adjoint.


\section{Background and Notational Conventions}\label{se.notation}

If $X$ and $G$ are compact Hausdorff spaces, then the \textit{join} of $X$ and $G$ is defined by
\bes \begin{aligned}
X * G := ( X \times G \times [0,1] ) / \sim, \hspace{.25 in}\\
(x, g, 0) \sim (x, h, 0) \hspace{.5 in} \forall g, h \in G, x \in X, \\
(x, g, 1) \sim (y, g, 1) \hspace{.5 in} \forall x, y \in X, g \in G, \\
\end{aligned} \ees
so that $X * G$ connects a copy of $X$ to a copy of $G$ within a continuum of copies of $X \times G$. The iterated joins of $G$ are denoted
\bes
G^{*n} = E_{n-1} G,
\ees
where the above connects $n$ copies of $G$ (or is defined using $n-1$ join procedures). If $G$ is a compact group which acts on $X$, then the diagonal action of $G$ on $X * G$ is given by $[(x, g, t)] \cdot h = [(x \cdot h, gh, t)]$, which extends to the iterated joins in a similar way.

As in \cite{BDH, joinfusion}, the join procedure generalizes to the noncommutative setting, with one key difference: the coaction induced on a join cannot be diagonal, so the boundary conditions of the join need to be adjusted to produce a compatible coaction. If $(H, \Delta)$ is a compact quantum group which coacts on a unital $C^*$-algebra $A$ through $\delta: A \to A \otimes_\text{min} H$, then the (equivariant) noncommutative join
\bes
\join{A}{\delta}{H} := \{f \in C([0,1], A \otimes_{\text{min}} H): f(0) \in \delta(A), f(1) \in \bC \otimes H \}
\ees
connects a copy of $A$ to a copy of $H$ in such a way that applying the rule
\be\label{eq:joinsaction}
\delta_\Delta f (t) := (\textrm{id}_A \otimes \Delta)(f(t))
\ee
produces a coaction of $(H, \Delta)$ on $\join{A}{\delta}{H}$. Moreover, if $\delta$ is free, meaning $\delta$ is injective and meets the Ellwood condition
\bes
\overline{\left\{ \sum_{\text{finite}} (a_i \otimes 1) \delta(b_i): a_i, b_i \in A \right \}} = A \otimes_{\text{min}} H
\ees
as in \cite{ellwood}, then $\delta_\Delta$ is also free. 

While the $\delta(A)$ boundary condition in the equivariant noncommutative join does not appear to match the topological case, there is no fundamental difference. That is, if $A = C(X)$, $H = C(G)$, and $\delta$ is dual to an action of $G$ on $X$, then $\join{A}{\delta}{H}$ is equivariantly isomorphic to $C(X * G)$. Note also that we have reversed the interval in the join from \cite{BDH}. The new convention will simplify the presentation in \Cref{se.type1} of the iterated joins of $H$. Moreover, we believe this arrangement is more in-tune with classical conventions, in that a null-homotopy begins with a general path and ends with a constant path, not the other way around. To simplify notation, we also assume henceforth that all tensor products are minimal.

We will primarily be concerned with group $C^*$-algebras. Let $\Gamma$ be a discrete group, and define
\be\label{eq:characterNC}
\Delta(\gamma) = \gamma \otimes \gamma
\ee
for each $\gamma \in \Gamma$. Further, let $\pi$ be a unitary representation of $\Gamma$ such that $\mathbb{C} \Gamma$ injects into the completion $C^*_\pi(\Gamma)$ and $\Delta$ extends linearly and continuously to a comultiplication on $H = C^*_\pi(\Gamma)$. For brevity, we will simply call $C^*_\pi(\Gamma)$ a $C^*$-completion of $\Gamma$ without repeating these conditions. If $\delta$ is a coaction of $H = C^*_\pi(\Gamma)$ on $A$, then for each $\gamma \in \Gamma$, we denote the $\gamma$-isotypic subspace of $(A, \delta)$ as
\bes
A_\gamma := \{a \in A: \delta(a) = a \otimes \gamma \}.
\ees
The isotypic subspaces produce a grading on $A$, and $\delta$ is free if and only if these subspaces satisfy the saturation property
\be\label{eq:saturation}
\forall \gamma \in \Gamma, \hspace{.3 in} 1 \in \overline{A_\gamma A_\gamma^*}.
\ee
Equivalently, the grading of $A$ by $\Gamma$ is strong (see \cite{montgomery, phillips}). This is not the most general form of a saturation condition, but it will be sufficient for our purposes -- see \cite{freeness} for discussion of the equivalence of saturation conditions and freeness in the general setting. Also, note that if $H = C(G) = C^*(\Gamma)$ for $G$ a compact abelian group and $\Gamma$ its discrete abelian dual, then a free coaction $\delta$ of $H$ on $A$ is dual to a free action $\beta$ of $G$ on $A$. In particular, \Cref{eq:characterNC} is essentially the declaration that $\gamma$ is a character, and the equation $\delta(\gamma) = a \otimes \gamma$ may be written in the equivalent form $\beta_g(a) = \gamma(g) a$. That is, in the classical setting, $A_\gamma$ is precisely the $\gamma$-spectral subspace.

If $A = H$ and $\delta = \Delta$, we may produce the join of $H$ with itself, and this procedure may be iterated, as in the classical case.  In fact, the iterated joins of $H$ have been crucial in the solution of both classical and noncommutative Borsuk-Ulam problems. We adopt two notational conventions below in an attempt to be consistent with the classical setting. Set
\bes
(E_0 H, \Delta_0) := (H, \Delta)
\ees
and
\bes
(E_n H, \Delta_n) := \left(\join{E_{n-1} H}{\Delta_{n-1}}{H}, (\Delta_{n-1})_\Delta\right),
\ees
where the subscript notation is used as in \Cref{eq:joinsaction}. That is, $E_n H$ is the join of $n + 1$ copies of $H$ with a free coaction $\Delta_n$ induced by the join procedure. We caution the reader that the equivariant join does not satisfy an obvious associativity property. Namely, there is always a single compact quantum group on the right, so the iteration must always proceed in a left-to-right fashion. Analogous to the classical case, we also let
\bes
H^{\circledast n} := E_{n-1} H
\ees
denote the join of $n$ copies of $H$. Similarly, we let $H^{\otimes n}$ and $H^{\oplus n}$ denote the tensor product or direct sum of $n$ copies of $H$, so that we may denote the trivial object using $0$ copies, as in $H^{\otimes 0} := \bC$ and $H^{\oplus 0} := \{0\}$.

\section{Type 1 and Dimension Invariants}\label{se.type1}

\Cref{conj:type12} Type 1 has been solved in some special cases, as in \cite{benfree, hajacindex}. Below we give a definition of a \textit{weak index} and \textit{strong index} in order to provide context for the methods used in both cases.

\begin{definition}
Fix a compact quantum group $(H, \Delta)$, and fix a class $\mathcal{C}$ of pairs $(A, \delta)$, where $\delta$ is a free coaction of $H$ on $A$. Suppose also that $(E_n H, \Delta_n) \in \mathcal{C}$ for all $n$ bigger than or equal to a fixed $N$. If $\alpha: \text{Free}(H) \to [-\infty, \infty]$ associates an extended real number to each pair $(A, \delta)$, then we call $\alpha$ a \textit{weak index} on $\mathcal{C}$ if it satisfies the properties (F), (D), and (I).
\begin{itemize}
\item (F) If $(A, \delta) \in \mathcal{C}$, then $\alpha(A, \delta)$ is finite.
\item (D) If $(A, \delta_A), (B, \delta_B) \in \mathcal{C}$ and there is a $(\delta_A, \delta_B)$-equivariant map $A \to B$, then $\alpha(B, \delta_B) \leq \alpha(A, \delta_A)$.
\item (I) The set $\{ \alpha(E_n H, \Delta_n): n \geq N \}$ is not bounded above.
\end{itemize}
If, in addition, $\mathcal{C}$ is closed under the equivariant join procedure, and $\alpha$ satisfies the following property (J), then $\alpha$ is called a \textit{strong index} on $\mathcal{C}$.
\begin{itemize}
\item (J) If $(A, \delta) \in \mathcal{C}$, then $\alpha(\join{A}{\delta}{H}, \delta_\Delta) = \alpha(A, \delta) + 1$.
\end{itemize}
\end{definition}

These properties stand for \qu{finiteness}, \qu{decreasing}, \qu{iteration}, and \qu{join}, respectively, where we note that property (J) is a strengthening of property (I). When $H$ and $\delta$ are understood, we will abuse notation by writing $\alpha(A)$ instead of $\alpha(A, \delta)$. We emphasize that a strong index is analogous, with a few tweaks, to the notion of an index in the classical setting (see \cite{volovikovindex}, for example). The specification of weak versus strong is meant to capture the minimal assumptions necessary for $\alpha$ to affirm the Type 1 Borsuk-Ulam conjecture for a fixed $(H, \Delta)$ and any $(A, \delta)$ in the class $\mathcal{C}$. 

Implicit in the definition of a weak index is the iteration procedure of \cite{benfree}, which is very much motivated by classical ideas. In its weakest form, the iteration procedure first posits that if an equivariant map $A \to \join{A}{\delta}{H}$ exists, then a quotient produces an equivariant map $A \to H = E_0 H$. It follows that $\{n: \exists \text{ equivariant } A \to E_n H\}$ is the set of all nonnegative integers, as
\bes
A \to E_n H \hspace{.1 in} \implies \hspace{.1 in} A \to E_{n+1} H
\ees 
via the composition
\bes
A \to \join{A}{\delta}{H} \to \join{E_n H}{\Delta_n}{H} = E_{n+1}H.
\ees
Thus, if $H$ coacts on $A$, and an index $\alpha$ exists on a class which includes $A$, then the supposed map $A \to \join{A}{\delta}{H}$ has produced a contradiction. Namely, $\alpha(A, \delta)$ is finite but is an upper bound of the set $\{ \alpha(E_n H, \Delta_n): n \geq 0 \}$, which is not bounded above.

The resolved cases of the Type 1 conjecture in \cite{benfree, hajacindex} ultimately rely on some form of weak index. In \cite{benfree}, coactions of $H = C(\bZ/n\bZ) = \bC[\bZ/n\bZ]$ were considered, and the arguments therein focused on spectral subspaces. Let $\gamma \in \bZ/n\bZ$ be a generator and consider the spectral subspace (that is, the isotypic subspace) $A_\gamma$. Then the \textit{spectral count}
\be\label{eq:sc}
c_\gamma(A) = \inf \left\{n \geq 0: \exists a_0, b_0, \ldots, a_n, b_n \in A_\gamma: \sum\limits_{i=0}^n a_i b_i^* \textrm{ is invertible} \right\}
\ee
is a weak index on the class of \textit{all} unital $C^*$-algebras with free actions of $\bZ/n\bZ$. In particular, finiteness of $c_\gamma(A)$ is due to the equivalence of freeness and saturation properties, and unboundedness of the set of $c_\gamma(E_nH)$ is due to the increasing connectivity of the iterated joins of $\bZ/n\bZ$ (see \cite[Proposition 4.4.3]{matousek} and \cite[Remark on p.68]{dold}). 

The invariant considered in \cite{hajacindex} is the local-triviality dimension of a coaction; we repeat the definition below.

\begin{definition}[Gardella-Hajac-Tobolski-Wu, \cite{hajacindex}]
Let $A$ be a unital $C^*$-algebra with a coaction $\delta$ of a compact quantum group $(H, \Delta)$. The \textit{local-triviality dimension} of $(A, \delta)$ is defined to be the infimum of the set of $d \geq 0$ such that there exist $H$-equivariant $*$-homomorphisms $\rho_0, \ldots, \rho_d: C_0((0,1]) \otimes H \to A$ such that $\sum\limits_{j=0}^d \rho_j(\textrm{id} \otimes 1) = 1$. Equivalently, the local-triviality dimension is the infimum of the set of $d \geq 0$ such that there exist $H$-equivariant, completely positive, contractive, order zero linear maps $\gamma_0, \ldots, \gamma_d: H \to A$ such that $\sum\limits_{j=0}^d \gamma_j(1) = 1$. 
\end{definition}
\begin{remark}
Note that the case $\inf(\varnothing) = \infty$ may occur. We will primarily use $\trivdim{A, \delta}{H}$, or abbreviated forms such as $\trivdim{A}{H}$, $\trivdim{A}{\phantom{H}}$, or $\trivdim{\delta}{\phantom{H}}$, to denote the triviality dimension. 
\end{remark}

The local-triviality dimension is used in \cite[Theorem 5.3]{hajacindex} to show that if $H$ admits a nontrivial classical subgroup $C(G)$ and $\trivdim{A}{C(G)}$ is finite, then the Type 1 conjecture holds for $A$. By passing from $H$ to the subgroup $C(G)$, and from $A$ to the largest abelian quotient $A / I = C(X)$, the argument essentially proves that $\trivdim{C(X)}{C(G)}$ is a strong index on the class of unital commutative $C^*$-algebras with coactions of $C(G)$ having finite local-triviality dimension.

It is not shown in \cite{hajacindex} that the local triviality dimension satisfies property (J) of a strong index when $A$ is noncommutative. However, the inequality $\trivdim{\join{A}{\delta}{H}, \delta_\Delta}{H} \leq \trivdim{A, \delta}{H} + 1$ is immediate in all cases, as we show below.

\begin{proposition}
Let $\delta$ be a coaction of $(H, \Delta)$ on $A$, and equip $\join{A}{\delta}{H}$ with the induced coaction $\delta_\Delta$. Then $\trivdim{\join{A}{\delta}{H}, \delta_\Delta}{H} \leq \trivdim{A, \delta}{H} + 1$. Consequently, $\trivdim{E_n H, \Delta_n}{H} \leq n$ for each $n$.
\end{proposition}
\begin{proof}
If $\trivdim{A, \delta}{H} = \infty$, then there is nothing to prove, so suppose $\trivdim{A, \delta}{H} = d$ and let $\gamma_0, \ldots, \gamma_d: H \to A$ be completely positive, contractive, order zero, $H$-equivariant linear maps such that $\sum\limits_{j=0}^d \gamma_j(1) = 1$. For $0 \leq j \leq d$, define $\widetilde{\gamma}_j: H \to \join{A}{\delta}{H}$ by
\bes
\widetilde{\gamma}_j(h)[s] := (1-s) \cdot \delta(\gamma_j(h)),
\ees
and define $\widetilde{\gamma}_{d+1}: H \to \join{A}{\delta}{H}$ by
\bes
\widetilde{\gamma}_{d+1}(h)[s] = s \cdot (1 \otimes h).
\ees
Then $\widetilde{\gamma}_0, \ldots, \widetilde{\gamma}_{d+1}$ are completely positive, contractive, order zero, $H$-equivariant linear maps with
\bes 
\sum_{j=0}^{d+1} \widetilde{\gamma}_j(1)[s] 	= \sum_{j=0}^d (1-s) \delta(\gamma_j(1)) + s \cdot (1 \otimes 1)
															= (1-s) \delta\left( \sum_{j=0}^d \gamma_j(1) \right) + s
															= (1 - s) \delta(1) + s
															= 1.
\ees
It follows that $\trivdim{\join{A}{\delta}{H}, \delta_\Delta}{H} \leq d + 1$. Finally, we note that since $\gamma_0 = \textrm{id}$ establishes $\trivdim{E_0H}{H} = 0$, we may induct to find $\trivdim{E_nH}{H} \leq n$.
\end{proof}

It would be of interest to know the most general circumstances under which equality holds, as in property (J) of a strong index. Similarly, for which $H$ is it guaranteed that all free coactions have finite local-triviality dimension? Operating once again in analogy with the abelian case, one might hope that if $H$ is finite-dimensional, then the local-triviality dimension of any free coaction is finite, and that this dimension also increases under the join procedure. However, our constructions below show that at least one of these properties must fail for many noncommutative $H$ of finite dimension. In fact, these claims are not restricted to local-triviality dimension.

We begin with a reformulation of the iterated join $E_n H$ as a set of functions from an $n$-simplex into $H^{\otimes n + 1}$ which meet boundary conditions.  We denote the standard $n$-dimensional simplex as $S(n) \subset \bR^{n}$. That is, $S(n)$ is the convex hull of $e_0 := \vec{0}$ and the standard basis vectors $e_1, \ldots, e_n$. Further, we denote the convex hull of vertices $e_{k_1}, \ldots, e_{k_m}$ by $S_{k_1 \ldots k_m}$. Similarly, a boundary condition on points of the subsimplex $S_{k_1\ldots k_m}$ is written as $C_{k_1\ldots k_m}$, so that we may declare
\bes
x \in S_{k_1\ldots k_m} \implies f(x) \in C_{k_1\ldots k_m}.
\ees
The trivial case is the $0$-join
\bes
E_0 H := H = C(S(0), H),
\ees
for which the condition $C_0 := H$ imposes no restrictions. Similarly, we have
\bes
E_1 H := \join{H}{\Delta}{H} = \{f \in C(S(1), H^{\otimes 2}): f(e_0) \in \Delta(H), f(e_1) \in \bC \otimes H\},
\ees
for which $C_0 := \Delta(H)$ and $C_1 := \bC \otimes H$ impose restrictions, but $C_{01} := H \otimes H$ does not. Finally, we may induct on this procedure: if $E_n H$ is given as
\be\label{eq:iteratedjoinsimplexform}
E_n H \cong \{f \in C(S(n), H^{\otimes n+1}): \textrm{ if } x \in S_{k_1 \ldots k_m} \textrm{ then } f(x) \in C_{k_1 \ldots k_m}\},
\ee
then we may produce boundary conditions $C^\prime$ for $E_{n+1} H$ as follows.

If $k_1, \ldots, k_m \leq n$, then define
\bes
C^{\prime}_{k_1 \ldots k_m} := (\text{id}^{\otimes n} \otimes \Delta)(C_{k_1 \ldots k_m})
\ees
and 
\bes
C^\prime_{k_1 \ldots k_m, n+1} := C_{k_1 \ldots k_m} \otimes H.
\ees
We must also consider the case when there are no $k_i$, for which we define
\bes
C^\prime_{n+1} := \bC^{\otimes n + 1} \otimes H.
\ees
Finally, we have that
\bes
E_{n+1} H \cong \{f \in C(S(n+1), H^{\otimes n + 2}): \textrm{ if } x \in S_{l_1 \ldots l_p}, \textrm{ then } f(x) \in C^\prime_{l_1 \ldots l_p}\},
\ees
so that by a recursive procedure, any $E_n H$ admits a presentation as the set of functions from $S(n)$ to $H^{\otimes n + 1}$ meeting conditions on each subsimplex. In particular, we note that in the \textit{interior} of any $m$-simplex $S_{k_1, \ldots k_{m+1}}$, only one boundary condition applies, and any $x \in \textrm{Int}(S_{k_1 \ldots k_{m+1}})$ is restricted to an embedded copy of $H^{\otimes m+1}$.

When $H = C^*_\pi(\Gamma)$ is given as a group $C^*$-algebra and $f$ is in the $\gamma$-isotypic subspace $(E_n H)_\gamma$, then by definition of the coaction and associated grading, this implies that for each $x \in S(n)$, $f(x) \in H^{\otimes n} \otimes \bC \gamma$. Combined with the boundary conditions above, every $x$ in the interior of an $m$-simplex $S_{k_1 \ldots k_{m+1}}$ is restricted to an embedded copy of $H^{\otimes m}$. That is, membership in the $\gamma$-isotypic subspace has removed one degree of freedom. We note that this applies when $m = 0$: the value of $f \in (E_n H)_\gamma$ at a vertex $e_p$ is determined up to a scalar in $\bC = H^{\otimes 0}$.

In what follows we will make repeated use of the following notion.

\begin{definition}\label{def.st-perf}
Let $\Gamma$ be a nontrivial discrete group, and fix a $C^*$-completion $\bC \Gamma \hookrightarrow C^*_\pi(\Gamma)$. The completion $C_\pi^*(\Gamma)$ is called {\it strongly perfect} if $\Gamma$ is perfect and every $\gamma\in \Gamma$ is in the same path connected component as $1$ in the unitary group of $C_\pi^*(\Gamma)$.
\end{definition}

Note that if $\Gamma$ is a finite nontrivial perfect group, then $\bC \Gamma$ is already a $C^*$-algebra, and it is strongly perfect. Therefore, the smallest strongly perfect $C^*_\pi(\Gamma)$ corresponds to the alternating group on $5$ generators. Further, if $C_\pi^*(\Gamma)$ is strongly perfect and $\pi$ dominates $\rho$, then $C^*_\rho(\Gamma)$ is also strongly perfect.

\begin{theorem}\label{thm:perfectunitaries}
Let $H = C_\pi^*(\Gamma)$ be strongly perfect. Then for any $\gamma \in \Gamma$ and for any $n$, there is a unitary element in $(E_n H)_\gamma$.
\end{theorem}
\begin{proof}
Fix $n$. An element in $E_n H$ may be considered as a function from $S(n)$ to $H^{\otimes n + 1}$ which meets boundary conditions on each subsimplex, as in \Cref{eq:iteratedjoinsimplexform}. For any group element $\gamma$, we will define a unitary element in $(E_n H)_\gamma$ via construction of maps $f^{(k)}_\gamma$ defined on the $k$-skeleta of $S(n)$, meeting the necessary boundary conditions. We note, however, that $f^{(k+1)}_\gamma$ will not usually extend $f^{(k)}_\gamma$.

The value of any function in $(E_n H)_\gamma$ on a particular vertex $e_0, \ldots, e_n$ must lie in a one-dimensional subspace of $H^{\otimes n + 1}$ (which depends on the vertex $e_i$). We may therefore define $f^{(0)}_\gamma$ by picking elements of these subspaces. Similarly, the boundary conditions of $(E_n H)_\gamma$ in the interior of each edge allow movement within an embedded copy of $H = C^*_\pi(\Gamma)$. By assumption, there is a path connecting $\gamma$ to $1$ within the unitaries of $H$, so we may extend $f^{(0)}_\gamma$ along edges to reach a unitary-valued $f^{(1)}_\gamma$ meeting the boundary conditions of $(E_nH)_\gamma$ on the $1$-skeleton.

Now, we may induct. Fix $1 \leq k \leq n - 1$, and for each $\zeta \in \Gamma$ suppose that we have selected a function $f_\zeta^{(k)}$ defined on the $k$-skeleton which meets the boundary conditions for $(E_n H)_\zeta$. If $\gamma \in \Gamma$ is fixed, then because $\Gamma$ is a perfect group, we may write $\gamma$ as a product of commutators $\gamma = \prod\limits_{j=0}^m \zeta_{2j} \cdot \zeta_{2j+1} \cdot \zeta_{2j}^{-1} \cdot \zeta_{2j+1}^{-1}$. Define 
\bes
g := \prod\limits_{j=0}^m f_{\zeta_{2j}}^{(k)} \cdot f_{\zeta_{2j+1}}^{(k)} \cdot f_{\zeta_{2j}}^{(k)^{-1}} \cdot f_{\zeta_{2j+1}}^{(k)^{-1}}
\ees
on the $k$-skeleton, noting that $g$ meets the boundary conditions of $(E_n H)_\gamma$. For any $(k+1)$-simplex $A \subseteq S(n)$, $g$ is defined at any boundary point of $A$, and we wish to extend it to the interior of $A$, retaining the fact that it is unitary-valued and meets boundary conditions for elements of $(E_n H)_\gamma$. We may move within (the unitaries of) an embedded copy of $H^{\otimes k+1}$ imposed by the boundary conditions. Now, as $g$ is represented as a product of commutators, its restriction to the boundary of $A$ is a trivial element in the group $\pi_k(\mathcal{U}(H^{\otimes k}))$, which is abelian. Therefore, $g$ may be extended to the interior of the subsimplex $A$, meeting the boundary conditions of $(E_n H)_\gamma$. Upon extending $g$ to each $(k+1)$-simplex, we have produced $f_{\gamma}^{(k+1)}$. Doing so for each $\gamma \in \Gamma$ completes the inductive step.
\end{proof}

\Cref{thm:perfectunitaries} applies when $\Gamma$ is a nontrivial finite perfect group, and the result is in stark contrast to the finite abelian case, as if $\gamma$ generates $\Gamma = \bZ/k\bZ$, then the spectral count \Cref{eq:sc} of $(E_n H)_\gamma$ increases without bound as $n$ increases. The following corollary shows that the discrepancy between the two cases is even more serious, through the use of universal constructions.

\begin{corollary}\label{cor:usualiterationfails}
Let $H = C^*(\Gamma)$ be a full group $C^*$-algebra which is strongly perfect. Fix a generating set $X$ of $\Gamma$ and let $F_X$ denote the free group on $X$. Then $A = C^*(F_X)$ admits a free coaction of $H$ such that for each $n \in \bZ^+$, there is an equivariant morphism $\phi_n: A \to E_n H$.
\end{corollary}
\begin{proof}
The full group $C^*$-algebra $A = C^*(F_X)$ may be presented as
\bes
A = C^*( U_x, x \in X \hspace{2 pt} | \hspace{2 pt} U_x U_x^* = 1 = U_x^* U_x),
\ees
which may be given a coaction $\delta: A \to A \otimes H$ defined by
\bes
\delta: U_x \mapsto U_x \otimes x.
\ees
The unital $*$-homomorphism $\delta$ exists, as $U_x \otimes x$ is unitary for each $x \in X$. It follows that $\delta$ is injective, and it is also free due to the saturation condition \Cref{eq:saturation}. Namely, there is a unitary in each isotypic subspace. Similarly, if $n$ is fixed, then by \Cref{thm:perfectunitaries}, for each $x \in X$ there is a unitary $y_x \in (E_n H)_x$. The universal property of $A$ shows there is a morphism $\phi_n: A \to E_n H$ defined by $\phi_n(U_x) = y_x$, which is equivariant as it preserves the isotypic subspaces.
\end{proof}

\begin{corollary}\label{cor:noweakindex}
Let $H = C^*(\Gamma)$ be a full group $C^*$-algebra which is strongly perfect. Then any class $\mathcal{C}$ for which there exists a weak index must exclude the $C^*$-algebraic constructions found in \Cref{cor:usualiterationfails}. In particular, if $\emph{Free}(H)$ denotes the class of all $(A, \delta)$ such that $\delta$ is a free coaction of $H$ on $A$, then there is no weak index defined on $\emph{Free}(H)$. 
\end{corollary}
\begin{proof}
Suppose $\alpha$ is a weak index which may be evaluated on the coaction of $C^*(\Gamma)$ on $A = C^*(F_X)$, as in \Cref{cor:usualiterationfails}. Then as there exist equivariant maps $\phi_n: A \to E_n H$ for all $n$, it follows that
\bes
\alpha(A) \geq \sup \{ \alpha(E_n H): n \in \mathbb{N} \}.
\ees 
This is a contradiction, as the right hand side is infinite by property (I), but the left hand side is finite by property (F).
\end{proof}

\Cref{cor:noweakindex} applies when $H = \bC \Gamma$, where $\Gamma$ is a finite nontrivial perfect group. Therefore, even in the finite-dimensional (quantum) setting, it is possible that free coactions of $(H, \Delta)$ are not classified by finiteness of a well-behaved invariant. However, for any particular choice of invariant, it is not clear which of the conditions of a weak index must fail. For example, do the constructions in \Cref{cor:usualiterationfails} have infinite local-triviality dimension, or does the local-triviality dimension fail to increase when an iterated join is applied?

\begin{question}
If $H$ is a fixed compact quantum group, does $\trivdim{E_n H, \Delta_n}{H} = n$ for all $n$?
\end{question}

It remains in the realm of possibility that all compact quantum groups, and in particular all finite-dimensional ones, have a positive answer. If this holds, then the local-triviality dimension would solve the Type 1 conjecture for any free coactions with $\trivdim{A, \delta}{H} < \infty$. However, \Cref{cor:noweakindex} would then show that, for certain finite-dimensional $H$, there can exist free coactions with infinite dimension, and those coactions would need to be handled separately. Below we consider a situation similar to \Cref{thm:perfectunitaries} for general $C^*$-algebras.

\begin{proposition}\label{prop:perfectcommun}
Let $\Gamma$ be a nontrivial discrete perfect group and fix a $C^*$-completion $C^*_\pi(\Gamma)$. Suppose $\delta$ is a coaction of $C^*_\pi(\Gamma)$ on a unital $C^*$-algebra $A$, such that every $\gamma$-isotypic subspace of $A$ contains a unitary element. If the group of path connected components of unitaries in $A$ is abelian, then the induced coaction of $C^*_\pi(\Gamma)$ on $\join{A}{\delta}{C^*_\pi(\Gamma)}$ also has a unitary element in each $\gamma$-isotypic subspace.
\end{proposition}
\begin{proof}
For each $\zeta \in \Gamma$, choose a unitary $U_\zeta$ in $A_\zeta$. Fix an arbitrary $\gamma \in \Gamma$. Since $\Gamma$ is perfect, we may write $\gamma = \prod\limits_{j=0}^m \zeta_{2j} \cdot \zeta_{2j+1} \cdot \zeta_{2j}^{-1} \cdot \zeta_{2j+1}^{-1}$. It follows that the unitary $W_\gamma := \prod\limits_{j=0}^m U_{\zeta_{2j}} \cdot U_{\zeta_{2j+1}} \cdot U_{\zeta_{2j}}^{-1} \cdot U_{\zeta_{2j+1}}^{-1}$ is also in $A_\gamma$. Since $W_\gamma$ is expressed as a commutator, there is a path $g(t)$ connecting $g(0) = W_\gamma$ to $g(1) = 1$ within the unitaries of $A$. It follows that the function
\bes
t \mapsto g(t) \otimes \gamma
\ees
represents a unitary element in the $\gamma$-isotypic subspace of $\join{A}{\delta}{C^*_\pi(\Gamma)}$.
\end{proof}

As used in \cite[Example 3]{taghavi}, $C_r^*(F_n)$ has stable rank one (see \cite[Theorems 1.1 and 3.8]{dykema} and the corrigendum \cite{dykemacorrected}), so the path connected components of its unitaries form an abelian group, and certain unitary elements may be immediately contracted to the identity. In the language of join constructions, the same reasoning shows that the Type 1 conjecture needs to see the difference between full and reduced $C^*$-algebras on which $H$ coacts.

\begin{proposition}\label{prop:fulltoreducedpath}
Suppose $\Gamma$ is a nontrivial discrete perfect group with finite generating set $X$, and define $H = C_r^*(\Gamma)$. Suppose there are well-defined, free coactions of $H$ on the full and reduced $C^*$-algebras $C^*(F_X)$ and $C_r^*(F_X)$ determined by the declaration that the unitary corresponding to the generator $x \in X$ is in the $x$-isotypic subspace. Then there is an equivariant morphism $\phi: C^*(F_X) \to \join{C^*_r(F_X)}{\delta}{H}$.
\end{proposition}
\begin{proof}
Write
\bes
C^*(F_X) = C^*( U_x, \hspace{2 pt} x \in X \hspace{2 pt} | \hspace{2 pt} U_x U_x^* = 1 = U_x^* U_x)
\ees
and
\bes
C^*_r(F_X) = C^*(F_X) / I, \hspace{.4 in} V_x := [U_x]_I.
\ees
By definition of the assumed coactions of $H = C_r^*(\Gamma)$ on both $C^*$-algebras, there is a unitary in each $\gamma$-isotypic subspace. Moreover, the group of path connected components of unitaries in $C_r^*(F_X)$ is abelian, so by \Cref{prop:perfectcommun}, there is a unitary $f_\gamma$ in each $\gamma$-isotypic subspace of $\join{C_r^*(F_X)}{\delta}{H}$. It follows that the map
\bes
U_x \mapsto f_x
\ees
is well-defined and equivariant from $C^*(F_X)$ to $\join{C_r^*(F_X)}{\delta}{H}$.
\end{proof}
\begin{remark}
The easiest case for which one can directly verify the hypotheses is when $\Gamma = X$ is a finite perfect group. In particular, the given coactions are then easily seen to be free (and in particular, injective).
\end{remark}

The full group $C^*$-algebra in the domain is crucial, as we need only find unitaries in a target to define a morphism, but the full group $C^*$-algebra does not have as many unitaries which may be connected to the identity. On the other hand, the reduced $C^*$-algebra in the codomain is also crucial, as many unitaries may be connected the identity, but there are not as many morphisms out of the reduced $C^*$-algebra.

\begin{question}
Is there a single $C^*$-algebra $C_\pi^*(F_X)$ which may replace both the full and reduced $C^*$-algebras in \Cref{prop:fulltoreducedpath}?
\end{question}

We now return to the classical case $H = C(G)$. As in \cite[Lemma 3.2 and Proposition 3.5]{alexbenjoin}, we may assume that $G$ is the $p$-adic integer group $\bZ_p$ for some prime $p$, and that $A = C(X)$ is commutative. Below, we find that the local-triviality dimension distinguishes objects that the spectral count does not. This follows because, in the same spirit as \Cref{prop:fulltoreducedpath}, the topological Borsuk-Ulam conjecture is also sensitive to the application of (different) quotients in the domain and codomain of a map.

For what follows, we adopt the following notational convention. Let $b$ be a divisor of $a$, and suppose $X$ admits a free action of $\bZ/a\bZ$, but $Y$ admits an action of $\bZ/a\bZ$ which factors through a free action of $\bZ/b\bZ$. Then a $\bZ/a\bZ$-equivariant map from $X$ to $Y$ will be called an $(a,b)$-equivariant map, or an $(a,b)$-map for short. For example, the natural quotient map
\bes
(\bZ/a\bZ)^{*m} \to (\bZ/b\bZ)^{*m}
\ees
is an $(a, b)$-map. Similarly, if $\omega$ is a primitive $a$th root of unity and $\omega^k$ has order $b$, then a map $\phi: \bS^{2n-1} \to \bS^{2n-1}$ satisfying $\phi(\omega \vec{z}) = \omega^k \phi(\vec{z})$ is an $(a, b)$-map.

Given a single $(a, b)$-map $\phi: \mathbb{S}^{2n - 1} \to \mathbb{S}^{2n-1}$ whose degree can be computed, it is useful to consider new $(a, b)$-maps of varying degrees which can be produced based on this information. Common constructions sometimes take a pleasant form when recast in a $C^*$-algebraic context, and we consider one such case below.

\begin{example}\label{ex:degree3shiftex}
Denote the complex coordinates of $\bS^3$ by $(z_1, z_2)$, and consider that $C(\bS^3)$ is given by the presentation
\bes
C(\bS^3) = C^*(Z_1, Z_2 \hspace{5 pt} | \hspace{5 pt} Z_1 Z_1^* + Z_2 Z_2^* = 1, \text{ generators are normal and commute}).
\ees
Moreover, $(z_1, z_2) \mapsto \begin{pmatrix} z_1 & z_2 \\ -z_2^* & z_1^* \end{pmatrix}$ shows $\bS^3$ is homeomorphic to $SU(2)$, and similarly $\mathcal{V} := \begin{pmatrix} Z_1 & Z_2 \\ -Z_2^* & Z_1^* \end{pmatrix}$ generates the cyclic group $K_1(C(\bS^3)) \cong \bZ$. If $\phi: \bS^3 \to \bS^3$ has degree $q$, then the dual $C^*$-morphism $\Phi$ is such that $\Phi(\mathcal{V}) := \begin{pmatrix} \Phi(Z_1) & \Phi(Z_2) \\ -\Phi(Z_2)^* & \Phi(Z_1)^* \end{pmatrix}$ is in the same $K_1$-class as $\mathcal{V}^q$. That is, the $K_1$-class of $\Phi(\mathcal{V})$ corresponds to the degree.

Let $\bZ/a\bZ$ act on $\bS^3$, and hence on $C(\bS^3)$, via $\alpha: (Z_1, Z_2) \mapsto (\omega Z_1, \omega Z_2)$ for a primitive $a$th root of unity $\omega$. The $K_1$-class of $Y := ((Z_1Z_1^*)^{a+1} + Z_2 Z_2^*)^{-1} \begin{pmatrix} Z_1^{a+1} & Z_2 \\ -Z_2^* & Z_1^{* a+1} \end{pmatrix}$ is $a+1$ (due to the degree computation in \cite[\S 3]{ya15degree}), and the $K_1$-class of $\mathcal{V}$ as above is 1. Both matrices have row $1$ entries in the $\omega$-eigenspace and row $2$ entries in the $\omega^{-1}$ eigenspace, so $Y^* \mathcal{V}$ is fixed by the action $\alpha$. Further, the $K_1$-class of $Y^* \mathcal{V}$ corresponds to $-(a+1) + 1 = -a$.

Now, suppose $b = a/r$ is a divisor of $a$ and $\Phi$ is of the form $\alpha(\Phi(Z_i)) = \omega^r \Phi(Z_i)$, so it is dual to an $(a, b)$ map with some degree $q$. Then the matrix
\bes
W := \Phi(\mathcal{V}) \cdot (Y^* \mathcal{V})
\ees 
has $K_1$ class corresponding to $q - a$. Moreover, $W$ is of the form $\begin{pmatrix} A & B \\ -B^* & A^* \end{pmatrix}$ for commuting normals $A$ and $B$ with $AA^* + BB^* = 1$; this follows since $\Phi(\mathcal{V})$, $Y$, and $\mathcal{V}$ have this form, the three matrices have entries in a commutative $C^*$-algebra, and $SU(2)$ is a group. Since the entries of $Y^* \mathcal{V}$ are fixed by $\alpha$, it follows that the entries of $W$ inherit the homogeneity properties of the entries of $\Phi(\mathcal{V})$: $\alpha(A) = \omega^r A$ and $\alpha(B) = \omega^r B$. Thus there is a unital $*$-homomorphism $\Psi: C(\bS^3) \to C(\bS^3)$ with $\Psi(Z_1) = A$ and $\Psi(Z_2) = B$, such that $\Psi$ is dual to an $(a, b)$-map with degree $q - a$.
\end{example}

\begin{theorem}\label{thm:nearmiss}
For any $k \geq 1$, $m \geq 4$, and $p \geq 2$, there is a $(p^{2k}, p^k)$-equivariant map
\[
\phi: (\bZ/p^{2k}\bZ)^{*m+1} \to (\bZ/p^k\bZ)^{*m}.
\]
Consequently, for any $j, k \geq 1$, $m \geq 4$, and $p \geq 2$, there is a $(p^{k \cdot 2^j}, p^k)$-equivariant map
\bes
\psi: (\bZ/p^{k \cdot 2^{j}} \bZ)^{*m + j} \to (\bZ/p^k\bZ)^{*m}.
\ees
\end{theorem}
\begin{proof}
We expand upon the reasoning\footnote{The authors remain very grateful to Robert Edwards for having shown them these techniques when pointing out mistakes in an early version of \cite{alexbenjoin}.} in \cite[\S 4]{alexbenjoin}. Write points $(z_1, z_2) \in \bS^{3}$ in polar coordinates $z_j = r_j u_j$, and define
\[
\psi: (r_1 u_1, r_2 u_2) \in \bS^3 \to (r_1 u_1^{p^k}, r_2 u_2^{p^k}) \in \bS^3.
\]
If $\omega$ is a primitive $p^{2k}$ root of unity, which defines a free $\bZ/p^{2k}\bZ$ action on $\bS^3$ by multiplication, then $\psi(\omega \vec{z}) = \omega^{p^k} \psi(\vec{z})$. Since $\omega^{p^k}$ is a primitive $p^k$ root of unity, $\psi$ is a $(p^{2k}, p^k)$ map. The degree of $\psi$ is $p^{2k}$, and the $(p^{2k}, p^k)$-map $\psi$ may be adjusted in degree by any multiple of $p^{2k}$, so there is a different $(p^{2k}, p^k)$-map $\gamma$ on $\bS^3$ which has degree zero. That is, $\gamma$ is homotopically trivial as an ordinary map. It follows that there is a $(p^{2k}, p^k)$-equivariant extension 
\[
\gamma: \bS^3 * \bZ/p^{2k}\bZ \to \bS^3
\]
defined by following a nullhomotopy across one cone over $\bS^3$ and extending equivariantly to the remaining cones. 

Since $(\bZ/p^{2k}\bZ)^{*5}$ is $4$-dimensional and $\bS^3 * \bZ/p^{2k}\bZ$ is $3$-connected, there is a $(p^{2k}, p^{2k})$-equivariant map $f: (\bZ/p^{2k}\bZ)^{*5} \to \bS^3 * \bZ/p^{2k}\bZ$ by the techniques of \cite{dold} (or direct construction). Similarly, as $\bS^3$ is $3$-dimensional and $(\bZ/p^k\bZ)^{*4}$ is $2$-connected, there is a $(p^k, p^k)$-equivariant map $g: \bS^3 \to (\bZ/p^k\bZ)^{*4}$.

Define $\rho$ to be the composition
\bes
(\bZ/p^{2k}\bZ)^{*5} \xrightarrow[(p^{2k}, p^{2k})]{f} \bS^3 * \bZ/p^{2k}\bZ \xrightarrow[(p^{2k}, p^k)]{\gamma}  \bS^3 \xrightarrow[(p^k, p^k)]{g} (\bZ/p^k\bZ)^{*4},
\ees
so that $\rho: (\bZ/p^{2k}\bZ)^{*5} \to (\bZ/p^k\bZ)^{*4}$ is a $(p^{2k}, p^k)$ map. If $m = 4$, define $\phi = \rho$, and if $m > 4$, join additional copies of the quotient map $\bZ/p^{2k}\bZ \to \bZ/p^k\bZ$ to $\rho$ in order to produce $\phi$.

Finally, by varying the value of $k$ and the number of joins, we may create a chain of equivariant maps
\bes
\cdots \hspace{.2 in} \xrightarrow[\hspace{.4 in}] \, (\bZ/p^{8k}\bZ)^{*m+3} \xrightarrow[(p^{8k}, p^{4k})]{} (\bZ/p^{4k}\bZ)^{*m+2} \xrightarrow[(p^{4k}, p^{2k})]{} (\bZ/p^{2k}\bZ)^{*m+1} \xrightarrow[(p^{2k}, p^k)]{} (\bZ/p^{k}\bZ)^{*m}
\ees
and form a $(p^{k \cdot 2^{j}}, p^k)$-equivariant map $\psi: (\bZ/p^{k \cdot 2^j} \bZ)^{*m + j} \to (\bZ/p^k\bZ)^{*m}$ as a composition of maps in the chain.
\end{proof}
\begin{remark}
There are no $\bZ_p$-equivariant maps $\bZ_p^{*m+1} \to \bZ_p^{*m}$ by an unpublished result of Edwards and Bestvina. Therefore, \Cref{thm:nearmiss} shows that the Edwards-Bestvina theorem is sensitive to the application of different finite quotients. 
\end{remark}

\begin{corollary}\label{cor:padiceigencount}
Let $\bZ_p$ act on $\bZ_p^{*m}$ diagonally, and let $\tau \in \widehat{\bZ_p}$ be a character. Then the spectral count of the iterated join is bounded by $c_\tau(C(\bZ_p^{*m})) \leq 1$ regardless of the choice of $m$.
\end{corollary}
\begin{proof}
Since the compact group $\bZ_p$ is the inverse limit of $\bZ/p^k\bZ$, it follows that the discrete group $\widehat{\bZ_p}$ is the direct limit of the finite cyclic groups $\widehat{\bZ/p^k\bZ} \cong \bZ/p^k \bZ$. Therefore, the order of $\tau$ is $p^k$ for some $k$, and $\tau$ is determined by some generating character $\rho \in \widehat{\mathbb{Z}/p^k\mathbb{Z}}$.

For $m \leq 4$, one may directly construct $2$ elements $a_0, a_1$ of the $\tau$-spectral subspace of $\bZ_p^{*m}$ whose square sum is invertible. First, note that the finite cyclic group $\bZ/p^{k}\bZ$ admits a $\rho$-equivariant map $\bZ/p^k \bZ \to \bS^1$, which is certainly homotopically trivial since its range is finite. It follows that there is a $\rho$-equivariant map $\bZ/p^k \bZ * \bZ/p^k \bZ \to \bS^1$. Taking the join of this map with itself produces a $\rho$-equivariant map $(\bZ/p^k \bZ)^{*4} \to \bS^1 * \bS^1 = \bS^3$. Finally, following through the quotient $\bZ_p \to \bZ/p^k \bZ$ gives a $\tau$-equivariant map $(\bZ_p)^{*4} \to \bS^3 \subseteq \bC^2$, which may be considered as $2$ elements of the $\tau$-spectral subspace whose square sum is invertible.

For $m > 4$, \Cref{thm:nearmiss} shows that we may create a chain of maps
\bes
\bZ_p^{*m} \to (\bZ/p^{k \cdot 2^{m-4}}\bZ)^{*m} \xrightarrow[(p^{k \cdot 2^{m-4}}, \hspace{2pt} p^k)]{\phantom{a}} (\bZ/p^{k}\bZ)^{*4} \to \bC^2
\ees
such that the composition again produces $2$ elements $a_0, a_1$ of the $\tau$-spectral subspace whose square sum is invertible. Therefore, in all cases, $c_\tau(C(\bZ_p^{*m})) \leq 1$.
\end{proof}

We note that \Cref{cor:padiceigencount} also answers \cite[Question 3.4]{alexbenjoin} negatively, as a positive answer would place a lower bound on the spectral counts of $\bZ_p^{*m}$ that increases with $m$. As such, since
\bes
1 = \sup_{\tau \in \widehat{\bZ_p}}\sup_{m \in \bZ^+} c_\tau(C(\bZ_p^{*m})) < \sup_{n \geq 0} \left( \trivdim{C(\bZ_p^{*n})}{C(\bZ_p)} \right) = \infty
\ees
by \Cref{cor:padiceigencount} and \cite[Theorem 2.7]{hajacindex}, the local-triviality dimension has distinguished objects which the spectral count could not. However, as in \cite{alexbenjoin}, we note that a map $X \to X * \bZ_p$ produces not only equivariant maps $X \to \bZ_p^{*n}$, but also $X \to \beta \bZ_p^{*\infty}$, where $\bZ_p^{*\infty}$ is the inductive limit of $\bZ_p^{*n}$ and $\beta$ denotes the Stone-Cech compactification. Whether or not the spectral count of $C(\beta \bZ_p^{*\infty})$ (defined in the very precise sense of \cite[Question 3.7]{alexbenjoin} due to non-continuity) is infinite, or whether or not maps of the form $X \to \bZ_p^{*\infty}$ can exist for other reasons (as in \cite[Question 3.6]{alexbenjoin}) is still of interest, as this could potentially prove the Type 1 conjecture in the classical case $H = C(G)$. That is, the Type 1 conjecture would hold even for spaces with infinite local-triviality dimension. Given the connections between the Type 1 conjecture and the Hilbert-Smith conjecture detailed in \cite{weakhs}, such a problem is likely to be quite difficult by association.

The proof of \Cref{cor:padiceigencount} also produces inverse limit spaces corresponding to chains of the form
\bes
\cdots \hspace{.2 in} \xrightarrow[\hspace{.4 in}] \, (\bZ/p^{8k}\bZ)^{*7} \xrightarrow[(p^{8k}, p^{4k})]{} (\bZ/p^{4k}\bZ)^{*6} \xrightarrow[(p^{4k}, p^{2k})]{} (\bZ/p^{2k}\bZ)^{*5} \xrightarrow[(p^{2k}, p^k)]{} (\bZ/p^{k}\bZ)^{*4}
\ees
for any fixed $k$. The space $X = \lim\limits_{\leftarrow n} \hspace{4 pt} (\bZ/p^{2^{n} \cdot k}\bZ)^{*(n+4)}$ admits a free action of $\bZ_p$, and it seems that $X$ is not clearly of finite local-triviality dimension, so that it should be investigated as a potential counterexample to the Type 1 conjecture. The inverse limit $X$ may be considered a way to form an \qu{infinite join} that is distinct from the usual inductive limit formulation (which is not compact).

We close this section with computations on the $\theta$-deformed spheres that concern degree-shifting procedures, local-triviality dimension, and the spectral count. The $\theta$-deformed spheres $C(\bS^{2k-1}_\theta)$ of \cite{matsumoto, natsume} are given by the presentations
\bes
C(\bS^{2k-1}_\theta) := C^*\left( Z_1, \ldots, Z_k \hspace{4 pt} \bigg\rvert \hspace{4 pt} [Z_j, Z_j^*] = 0, \hspace{4 pt} Z_j Z_l = e^{2 \pi i \theta_{lj}}Z_l Z_j, \hspace{4 pt} \sum\limits_{j=1}^k Z_j Z_j^* = 1\right),
\ees
indexed by antisymmetric matrices $\theta \in M_k(\bR)$. Each $\theta$-deformed sphere has $K_1(C(\bS^{2k-1}_\theta)) \cong \bZ$, and in fact, $K_1(C(\bS^{2k-1}_\theta))$ is generated by a unitary matrix $\mathcal{V}_\theta$ of dimension $2^{k-1} \times 2^{k-1}$, where each entry of $\mathcal{V}_\theta$ is a $*$-monomial. Moreover, the entries (or rather, the coefficients) of $\mathcal{V}_\theta$ vary continuously in $\theta$. With notation as in \cite[(4.8)]{bentheta}, if 
\bes
R_\omega: Z_i \mapsto \omega_i Z_i
\ees
denotes a rotation action for $m$th roots of unity $\omega_1, \ldots, \omega_k$, then there are diagonal matrices $A_\omega, B_\omega \in U_{2^{k-1}}(\bC)$, whose entries are independent of $\theta$, such that $A_\omega^m = B_\omega^m = I$ and
\be\label{eq:rotconj}
R_\omega(\mathcal{V}_\theta) = A_\omega \mathcal{V}_\theta B_\omega.
\ee
We note that while the claim is made in \cite{bentheta} only for primitive roots of unity, there is no need for that assumption (primitivity was only used starting in later results, such as \cite[Theorem 4.11 and Corollary 4.12]{bentheta}). We consider non-primitive roots of unity in the following theorem.

\begin{theorem}\label{thm:Ktheoryexhaustion}
Let $\omega_1, \ldots, \omega_k$ be primitive $m$th roots of unity, and let $\gamma_1, \ldots, \gamma_k$ be (possibly nonprimitive) $m$th roots of unity. Suppose that $\Phi_1, \Phi_2: C(\bS^{2k-1}_\theta) \to C(\bS^{2k-1}_\zeta)$ are $(R_\gamma, R_\omega)$-equivariant unital $*$-homomorphisms. Then the integers $[\Phi_1(\mathcal{V}_\theta)]_{K_1}$ and $[\Phi_2(\mathcal{V}_\theta)]_{K_1}$ are congruent modulo $m$.
\end{theorem}
\begin{proof}
It suffices to prove that $[\Phi_1(\mathcal{V}_\theta)\Phi_2(\mathcal{V}_\theta)^{*}]_{K_1}$ is a multiple of $m$. First, note that since each $\gamma_i$ is an $m$th root of unity (though perhaps not primitive), we have that the matrices $A_\gamma$ and $B_\gamma$ fulfilling the role of \Cref{eq:rotconj} for the rotation $R_\gamma$ satisfy $A_\gamma^m = I = B_\gamma^m$. Also, since the $\Phi_i$ are $(R_\gamma, R_\omega)$-equivariant unital $*$-homomorphisms, and the matrices $A_\gamma$ and $B_\gamma$ have scalar entries, it follows that
\bes\begin{aligned}
R_\omega( \Phi_1(\mathcal{V}_\theta) \Phi_2(\mathcal{V}_\theta)^*) 	&= \Phi_1(R_\gamma(\mathcal{V}_\theta))\Phi_2(R_\gamma(\mathcal{V}_\theta^*)) \\
																	&= \Phi_1(A_\gamma \mathcal{V}_\theta B_\gamma)\Phi_2(B_\gamma^* \mathcal{V}_\theta^* A_\gamma^*)\\
																	&= A_\gamma \Phi_1(\mathcal{V}_\theta) \Phi_2(\mathcal{V}_\theta)^* A_\gamma^*.
\end{aligned} \ees
That is, $W := \Phi_1(\mathcal{V}_\theta) \Phi_2(\mathcal{V}_\theta)^*$ satisfies $R_\omega(W) = A_\gamma W A_\gamma^*$. Since $A_\gamma \in U_{2^{k-1}}(\bC)$ satisfies $A_\gamma^m = I$, and $R_\omega$ is such that each $\omega_i$ is a \textit{primitive} $m$th root of unity, it follows from \cite[Theorem 4.11]{bentheta} that $[W]_{K_1}$ is a multiple of $m$.
\end{proof}

The case of \Cref{thm:Ktheoryexhaustion} in which each $\gamma_i$ is a \textit{primitive} $m$th root of unity is covered by \cite[Corollary 4.12]{bentheta}. In fact, in this case, the computed $K_1$-classes are both congruent to $1$ modulo $m$, and in particular they are nontrivial if $m \geq 2$. 

Restricting to the commutative case $\theta = 0 = \zeta$ and applying the Chern character recovers a (likely well-known) topological Borsuk-Ulam result for rotation actions on spheres, where said actions may or may not be free. We list it here for completeness.

\begin{corollary}\label{cor:degreeexhaustion}
Let $\omega_1, \ldots, \omega_k$ be primitive $m$th roots of unity, and let $\gamma_1, \gamma_2, \ldots, \gamma_k$ be primitive $n_i$th roots of unity, where each $n_i$ is a divisor of $m$. Equip a domain sphere $\bS^{2k-1}$ with the rotation action $(z_1, \ldots, z_k) \mapsto (\omega_1 z_1, \ldots, \omega_k z_k)$ and a codomain sphere $\bS^{2k-1}$ with the rotation action $(z_1, \ldots, z_k) \mapsto (\gamma_1 z_1, \ldots, \gamma_k z_k)$. If $\phi: \bS^{2k-1} \to \bS^{2k-1}$ is an equivariant, continuous map, then 
\bes
\emph{deg}(\phi) \equiv \cfrac{m^k}{n_1 \cdots n_k} \textrm{ mod } m.
\ees
\end{corollary}
\begin{proof}
Write $(z_1, \ldots, z_k) \in \bS^{2k-1}$ in polar coordinates $z_j = r_j u_j$, and consider maps of the form
\be\label{eq:polarpower}
\psi_b: (r_1 u_1, \ldots, r_k u_k) \mapsto (r_1 u_1^{b_1}, \ldots, r_k u_k^{b_k})
\ee
for $b_1, \ldots, b_k \in \bZ^+$, so that if $\psi_b(z_1, \ldots, z_k) = (w_1, \ldots, w_k)$, it follows that $\psi_b(\omega_1 z_1, \ldots, \omega_k z_k) = (\omega_1^{b_1} w_1, \ldots, \omega_k^{b_k} w_k)$. Since the $\omega_i$ are primitive $m$th roots of unity, the $\gamma_i$ are primitive $n_i$th roots of unity, and $n_i$ divides $m$, we may select $b_i$ such that $\omega_i^{b_i} = \gamma_i$. That is, $\psi_b$ is equivariant for the given actions. For such a choice of $b_i$, we will have $b_i \equiv \frac{m}{n_i} \hspace{1pt} \textrm{ mod } m$, and hence $\textrm{deg}(\psi_b) = b_1\cdots b_k \equiv \cfrac{m^k}{n_1 \cdots n_k} \textrm{ mod } m$.

Suppose that $\phi$ is another equivariant map on $\bS^{2k-1}$ for the given actions. Then \Cref{thm:Ktheoryexhaustion}, in the case $\theta = 0 = \zeta$, shows that the degrees of $\psi_b$ and $\phi$ are congruent modulo $m$.
\end{proof}
\begin{remark}
Similar claims can be made in the $\theta$-deformed case, so long as one can provide at least one example of an equivariant map $C(\bS^{2n-1}_\theta) \to C(\bS^{2n-1}_\zeta)$ for the given actions and then compute that map's effect on odd $K$-theory. Since a polar decomposition does exist in the $\theta$-deformed case, this is sometimes possible using a formula analogous to \Cref{eq:polarpower} to define a map. However, such computations depend quite heavily on the deformation parameter.
\end{remark}

The case $n_i = m = 2$ of \Cref{cor:degreeexhaustion} is the traditional odd-degree formulation of the Borsuk-Ulam theorem (restricted to spheres of odd dimension). We note that it is sometimes possible to obtain $\textrm{deg}(\phi) = 0$ in \Cref{cor:degreeexhaustion}, so the result does not always imply that $\phi$ is homotopically nontrivial; this is exactly the pathology abused in \Cref{ex:degree3shiftex} and the proof of \Cref{thm:nearmiss}. Rather, \Cref{cor:degreeexhaustion} implies that the degrees of equivariant maps that may be obtained using this degree-shifting technique exhaust all possible options.

While \Cref{thm:Ktheoryexhaustion} shows that some degree properties of equivariant morphisms $C(\bS^{2n-1}_\theta) \to C(\bS^{2n-1}_\zeta)$ are restricted in such a way that is independent of the deformation parameter, we note that the local-triviality dimension and spectral count do not necessarily remain the same when the deformation parameter changes. The Borsuk-Ulam theorem shows that the antipodal action on the  commutative sphere $C(\bS^{2n-1})$ has local-triviality dimension $2n - 1$ and spectral count $n - 1$, but \cite[Theorem 3.15]{bentheta} implies that the spectral count can become lower under a $\theta$-deformation. A small adjustment to that result also applies to the local-triviality dimension, so long as we choose a $\theta$-deformed sphere carefully.

\begin{proposition}\label{prop:deformless}
Let $C(\bS^{2n-1}_\theta)$ be the $(2n-1)$-dimensional $\theta$-deformed sphere whose generators $Z_j$ pairwise anticommute. Then the antipodal action $Z_j \mapsto -Z_j$ has both local-triviality dimension and spectral count equal to $1$.
\end{proposition}
\begin{proof}
Decomposing the normal generators as $Z_j = X_j + i Y_j$ shows that the $X_j$ pairwise anticommute and the $Y_j$ pairwise anticommute. The self-adjoint elements
\bes
a_0 := X_1 + \ldots + X_n \hspace{1 in} a_1 := Y_1 + \ldots + Y_n
\ees
are odd (that is, in the $-1$ spectral subspace of the action) and have
\bes
a_0^2 + a_1^2 = (X_1^2 + \ldots + X_n^2) + (Y_1^2 + \ldots + Y_n^2) = Z_1Z_1^* + \ldots + Z_nZ_n^* = 1.
\ees
This implies that the spectral count of $C(\bS^{2n-1}_\theta)$ is at most $1$. Next, it follows from \cite[Propositions 3.9 and 3.11]{bentheta} that $C(\bS^{2n-1}_\theta)$ has no odd left-invertible elements, so the spectral count is exactly $1$.  

We now consider local-triviality dimension. Define $b_i := a_i |a_i|^{-1/2}$ using the functional calculus (which is possible since the power $1/2$ is less than $1$), so that $b_0$ and $b_1$ are self-adjoint odd elements with $|b_0| + |b_1| = 1$. Letting $C(\bZ/2\bZ) = \bC + \bC \omega$ for a self-adjoint odd unitary $\omega$, there are completely positive, contractive, order zero maps $\gamma_i: C(\bZ/2\bZ) \to C(\bS^{2n-1}_\theta)$ defined by
\bes
\gamma_i(1) = |b_i| \hspace{1 in} \gamma_i(\omega) = b_i,
\ees
and hence $\trivdim{C(\bS^{2n-1}_\theta)}{C(\bZ/2\bZ)} \leq 1$. Since there is an equivariant map $C(\bS^{2n-1}_\theta) \to C(\bS^1)$ from annihilation of all but one generator, $\trivdim{C(\bS^{2n-1}_\theta)}{C(\bZ/2\bZ)} \geq \trivdim{C(\bS^1)}{C(\bZ/2\bZ)} = 1$.
\end{proof}

The local-triviality dimension as currently defined counts how many completely positive contractive, equivariant, order zero maps $\gamma_0, \ldots, \gamma_{n}: H \to A$ are needed to achieve $\sum \gamma_i(1) = 1$. We note, however, that when $A = C(X)$ and $H = C(G)$ are commutative, this condition is equivalent to the demand that $\sum \gamma_i(1) := \phi$ is invertible, as we may always rescale the maps by the central positive invariant element $\phi^{-1}$. Further, since the noncommutative Borsuk-Ulam theorem proved in \cite[Theorem 5.3]{hajacindex} ultimately reduces to the commutative setting, there appears to be no meaningful distinction between the two possible definitions in Borsuk-Ulam contexts. In the quantum setting, relaxing the condition on $\sum \gamma_i(1)$ may allow more coactions to be finite-dimensional, or it may pose problems with the potential growth of the dimension of $E_nH$ in the parameter $n$.

\begin{question}
In the context of noncommutative Borsuk-Ulam problems, is there a meaningful distinction between the condition that $\sum \gamma_i(1) = 1$ or the condition that $\sum \gamma_i(1)$ is invertible in the definition of local-triviality dimension?
\end{question}

Regardless of the choice, it is possible for the local-triviality dimension to decrease in a deformation. In fact, if the definition is adjusted, then a computation similar to \Cref{prop:deformless} would extend to more $\theta$-deformed spheres, as in \cite[Theorem 3.15]{bentheta}. Similarly, an adjusted definition would allow the claim \lq\lq a $\bZ/2\bZ$ action is free if and only if it has finite (adjusted) local-triviality dimension\rq\rq\hspace{0pt}. This is consistent with the observation that the spectral count is sufficient to prove a Borsuk-Ulam theorem for finite cyclic group actions \cite[Corollary 2.4 Alternative Proof A]{benfree} without explicitly reducing to commutative quotients. From this point of view, it is somewhat bizarre that the local-triviality dimension as currently defined might only capture the $\bZ/2\bZ$ case when applied to a commutative quotient $A / I \cong C(X)$ instead of $A$ itself.

\section{Type 2 and Cleftness}\label{se.type2}

In \cite[Corollary 2.5]{alexbenjoin}, it is shown that any compact quantum group $H = C(\bT^X)$ corresponding to a torus admits a free coaction on some unital $C^*$-algebra $A$ which gives a counterexample to \Cref{conj:type12} Type 2. That is, there exists an equivariant morphism $H \to \join{A}{\delta}{H}$. On the other hand, if $H$ is a compact quantum group which admits an embedding $K \hookrightarrow H$ of a nontrivial finite-dimensional compact quantum group $K$, then there are no counterexamples among free coactions of $H$ by \cite[Theorem 2.7]{alexbenjoin}. Since a compact abelian group $G$ is connected if and only if $\Gamma = \widehat{G}$ is torsion-free, this suggests that in the abelian case, the classification of when Type 2 counterexamples exist may come down to a connectedness or torsion property.

Let $\Gamma$ be an abelian discrete group. Below we show that one way to achieve counterexamples to Type 2 is to connect a morphism of $C^*(\Gamma)$ of the form $\gamma \mapsto \begin{pmatrix} \gamma \\ & \gamma^{-1} \end{pmatrix}$ to the trivial representation $\gamma \mapsto \begin{pmatrix} 1 \\ & 1 \end{pmatrix}$ in a stabilized fashion. In fact, it suffices to contract an infinite direct sum of such representations in a stabilized fashion, which is a more relaxed condition. We denote endomorphisms of $C^*(\Gamma)$ that are induced by endomorphisms of $\Gamma$ as follows: 
\bes
\textrm{id}: \gamma \to \gamma, \hspace{.4 in} \text{inv}: \gamma \to \gamma^{-1}, \hspace{.4 in} \text{triv}: \gamma \to 1, \hspace{.4 in} \text{squ}: \gamma \to \gamma^2.
\ees
For any unital $C^*$-algebra $B$, we give the space of unital $*$-homomorphisms $\textrm{Hom}(C^*(\Gamma), B)$ the pointwise norm topology. Also, if $\cH$ is an infinite-dimensional Hilbert space, we frequently abuse notation by making such identifications as $\textrm{id}^{\oplus \cH} = \textrm{id}^{\oplus \cH} \oplus \textrm{id}^{\oplus \cH}$, since $\cH \oplus \cH \cong \cH$.

\begin{theorem}\label{thm:pathcontraction}
Let $\Gamma$ be a discrete abelian group such that $C^*(\Gamma)$ satisfies following representation connectedness property: there exists an infinite-dimensional Hilbert space $\cH$ such that there is a continuous path within $\emph{Hom}(C^*(\Gamma), B(\cH) \otimes C^*(\Gamma))$ connecting $\emph{id}^{\oplus \cH} \oplus \emph{inv}^{\oplus \cH} \oplus \emph{triv}^{\oplus \cH}$ to $\emph{triv}^{\oplus \cH}$. Then there exists a Type 2 counterexample for the compact quantum group $C^*(\Gamma)$ coacting on $A := B(\cH) \otimes C^*(\Gamma)$ through the right tensorand.
\end{theorem}
\begin{proof}
We note that because $C^*(\Gamma)$ is commutative, any element of $\bC \otimes C^*(\Gamma)$ is in the center of $B(\cH) \otimes C^*(\Gamma)$. As such, let $\phi_t: C^*(\Gamma) \to B(\cH) \otimes C^*(\Gamma)$ be a path of morphisms connecting $\phi_0 = \textrm{id}^{\oplus \cH} \oplus \textrm{inv}^{\oplus \cH} \oplus \textrm{triv}^{\oplus \cH}$ to $\phi_1 = \textrm{triv}^{\oplus \cH}$, and define
\bes
\psi_t(\gamma) := (1 \otimes \gamma) \cdot \phi_t(\gamma) = \phi_t(\gamma) \cdot (1 \otimes \gamma).
\ees
It follows that $\psi_t$ defines a homomorphism from $\Gamma$ to the unitary group of $B(\cH) \otimes C^*(\Gamma)$, and hence also defines an element of $\textrm{Hom}(C^*(\Gamma), B(\cH) \otimes C^*(\Gamma))$, that varies continuously in $t$. The path connects $\psi_0 = \textrm{squ}^{\oplus \cH} \oplus \textrm{triv}^{\oplus \cH} \oplus \textrm{id}^{\oplus \cH}$ to $\psi_1 = \textrm{id}^{\oplus \cH}$.

We now show that we may use the above paths to connect the morphism $\textrm{id}^{\oplus \cH}$ to the morphism $\textrm{triv}^{\oplus \cH}$ within $\textrm{Hom}(C^*(\Gamma), B(\cH) \otimes C^*(\Gamma))$. Beginning with $\textrm{id}^{\oplus \cH}$, we may apply the path $\psi$ to connect $\textrm{id}^{\oplus \cH}$ to $\textrm{squ}^{\oplus \cH} \oplus \textrm{triv}^{\oplus \cH} \oplus \textrm{id}^{\oplus \cH}$. Next, we focus on the middle summand, so that we may apply $\phi$ to obtain a path connecting $\textrm{squ}^{\oplus \cH} \oplus \textrm{triv}^{\oplus \cH} \oplus \textrm{id}^{\oplus \cH}$ to $\textrm{squ}^{\oplus \cH} \oplus \left( \textrm{id}^{\oplus \cH} \oplus \textrm{inv}^{\oplus \cH} \oplus \textrm{triv}^{\oplus \cH} \right) \oplus \textrm{id}^{\oplus \cH}$. Reorganizing using a path of unitary conjugations leads to $\left(\textrm{squ}^{\oplus \cH} \oplus \textrm{triv}^{\oplus \cH} \oplus \textrm{id}^{\oplus \cH} \right) \oplus \textrm{triv}^{\oplus \cH} \oplus \textrm{inv}^{\oplus \cH}$, which may be connected using $\psi$ to $\textrm{id}^{\oplus \cH} \oplus \textrm{triv}^{\oplus \cH} \oplus \textrm{inv}^{\oplus \cH}$. Finally, we may apply $\phi$ once more to connect this morphism to $\textrm{triv}^{\oplus \cH}$.

A path connecting the morphisms $\textrm{id}^{\oplus \cH}$ and $\textrm{triv}^{\oplus \cH}$ within $\textrm{Hom}(C^*(\Gamma), B(\cH) \otimes C^*(\Gamma))$ shows that if $A = B(\cH) \otimes C^*(\Gamma)$ is given the right tensorand coaction of $C^*(\Gamma)$, then an equivariant map $C^*(\Gamma) \to A$ may be connected to a one-dimensional representation of $C^*(\Gamma)$. By the construction in \cite[Lemma 2.2]{alexbenjoin}, a Type 2 counterexample exists.
\end{proof}

Consider the case $\Gamma = \bZ$, generated by $\gamma$. It is possible to connect $\textrm{id} \oplus \textrm{inv}$ to $\textrm{triv}^{\oplus 2}$ without any stabilization, as in the $C^*$-algebra $M_2(\bC) \otimes C^*(\Gamma)$, the unitary $\gamma \oplus \gamma^{-1}$ may be connected to $I_2$ using an explicit continuous path $\phi_t$. Therefore, the path demanded in \Cref{thm:pathcontraction} may be written as an infinite direct sum built up from $\phi_t$ and trivial representations. Following the proof then gives a more explicit form of the counterexample in \cite[Theorem 2.3]{alexbenjoin}. 

Generally speaking, we expect that Type 2 counterexamples will need to use infinite-dimensional constructions. For example, consider the case $\Gamma = \bZ$ once again, generated by $\gamma$. It is impossible to connect $\textrm{id}^{\oplus n}$ to $\textrm{triv}^{\oplus n}$ within $\textrm{Hom}(C^*(\Gamma), M_n(\bC) \otimes C^*(\Gamma))$. In particular, applying the determinant shows that such a path would place $\gamma^n$ in the same path component as $1$ in the unitary group of $C^*(\Gamma) \cong C(\bS^1)$, contrary to elementary facts about the fundamental group of $\bS^1$. However, when $M_n(\bC)$ is replaced by $B(\cH)$, this obstruction vanishes. Note also that the counterexamples constructed for Type 2 so far have all been of the form $A = B \otimes C(G)$ for some $C^*$-algebra $B$, where $C(G)$ coacts on $A$ through the right tensorand only.

Types 1 and 2 of \Cref{conj:type12} overlap along the question \lq\lq does there exist an equivariant map $\phi: H \to \join{H}{\Delta}{H}$?\rq\rq\hspace{0pt} When $H = C(G)$ is classical, it is known that $\phi$ cannot exist because $G$ is not contractible. When $H$ has a counit, \cite[Lemma 2.4]{BDHrevisited} implies that existence of $\phi$ is equivalent to the claim that the identity map on $H$ may be connected within the endomorphisms of $H$ to a character. Thus, when $H$ does not necessarily have a counit, one may view existence of $\phi$ as a noncommutative analogue of the claim that $H$ is contractible (or more precisely, that the quantum group underlying the Hopf $C^*$-algebra $H$ is contractible). This idea is implicit as far back as \cite{BDH}.

With this perspective, it is natural to seek obstructions to contractibility using well-established invariants, as in the following result. For background on the Baum-Connes conjecture we refer the reader to the introductory paper \cite{val} as well as \cite{bc}, where the conjecture originates.

\begin{proposition}\label{pr.bc}
  Let $(H,\Delta)$ be the reduced $C^*$-algebra $C^*_r(\Gamma)$ of a discrete group that satisfies the Baum-Connes conjecture and which has non-trivial rational homology. Then there is no equivariant morphism $\phi: H \to \join{H}{\Delta}{H}$.
\end{proposition}
\begin{proof}
  \cite[Theorem 2.7]{alexbenjoin} proves that \Cref{conj:type12} Type 2 holds when $\Gamma$ has torsion, so we may as well assume $\Gamma$ is torsion-free.

  Existence of $\phi$ implies that we can find a path connecting an endomorphism
  \begin{equation*}
   \psi_0:C_r^*(\Gamma)\to C_r^*(\Gamma)\cong C_r^*(\Gamma)\otimes \bC\subset C_r^*(\Gamma)^{\otimes 2} 
  \end{equation*}
 which respects the $\Gamma$-grading to one of the form
  \begin{equation*}
    \psi_1:C_r^*(\Gamma)\to \mathrm{Im}(\Delta)\subset C_r^*(\Gamma)^{\otimes 2}. 
  \end{equation*}
  Restricting attention to the group $\Gamma$ inside $C_r^*(\Gamma)$, the maps given take the form
  \begin{equation*}
    \Gamma\ni \gamma\mapsto \varphi_0(\gamma)\gamma\in C_r^*(\Gamma)
  \end{equation*}
  and
  \begin{equation*}
    \Gamma\ni \gamma\mapsto \varphi_1(\gamma)\gamma\otimes\gamma\in \mathrm{Im}(\Delta)
  \end{equation*}
for characters $\varphi_i:\Gamma\to \bS^1$. This, in turn, implies that the maps induced by $\psi_i$ (and denoted by the same symbols here) coincide on $K$-theory. 

In the torsion-free case, the Baum-Connes conjecture implements an isomorphism
\begin{equation*}
  \mu:K_*(B\Gamma)\to K_*(C_r^*(\Gamma))
\end{equation*}
between $K_*(C_r^*(\Gamma))$ and the $K$-homology $K_*(B\Gamma)$ of the classifying space of $\Gamma$ (see e.g. \cite[Introduction]{val}). Extending coefficients to the field $\bQ$ turns
\begin{equation*}
 K_*(B\Gamma)_{\bQ}:=K_*(B\Gamma)\otimes_{\bZ}\bQ 
\end{equation*}
into a rational coalgebra, co-augmented in the sense that it is a rational coalgebra with a distinguished group-like element $g$ coming from the inclusion $1\to \Gamma$ (see \cite[$\S$4]{brd}). 

In turn, for $i=0,1$ we have
\begin{equation*}
  ch:K_i(B\Gamma)\otimes_{\mathbb{Z}}\mathbb{Q} \cong \bigoplus_{n} H_{i+2n}(\Gamma,\mathbb{Q})
\end{equation*}
by \cite[Remark 4.2.13 (2)]{val}. (The result is stated there over the complex numbers, but there is no material distinction: the essence is that homology of the group is isomorphic to $K$-homology of the classifying space up to torsion. See also \cite[Lemma A.4.2]{val}). The isomorphism is implemented by the {\it Chern character}, which is functorial in $\Gamma$; this follows for instance from its description in \cite[$\S$4.2]{jak}. Therefore, $ch$ respects the diagonal map $\Delta:\Gamma\to \Gamma^{\times 2}$ and the obvious maps $1\to \Gamma\to 1$.

All together, $ch$ is an isomorphism between $K_*(B\Gamma)_{\bQ}$ and the rational homology $H_*(\Gamma,\bQ)$ of $\Gamma$, regarded as a (graded) rational coalgebra co-augmented by the grouplike element
\begin{equation*}
  g=1\in\bQ\cong H_0(\Gamma,\bQ). 
\end{equation*}
Finally, the existence of $\phi$ now implies that up to the automorphisms induced by $\varphi_i$, $i \in \{0,1\}$, the maps
\begin{equation*}
  H_*(\Gamma,\bQ)\ni x\mapsto x\otimes 1\in H_*(\Gamma,\bQ)^{\otimes 2}
\end{equation*}
and 
\begin{equation*}
  H_*(\Gamma,\bQ)\ni x\mapsto \Delta(x)\in H_*(\Gamma,\bQ)^{\otimes 2}
\end{equation*}
are equal. Applying the counit $\varepsilon:H_*(\Gamma,\bQ)\to \bQ$ to the left tensorand in $H_*(\Gamma,\bQ)$, we conclude that $H_*(\Gamma,\bQ)$ is trivial (more precisely, $\bQ\cong H_*(\Gamma,\bQ)$), contradicting our assumption to the contrary.
\end{proof}
\begin{remark}
  The conclusion of \Cref{pr.bc} is proven for free groups in \cite[Theorem 3.8]{BDHrevisited}. Since free groups satisfy the Baum-Connes conjecture (e.g. because they are one-relator groups \cite{bbv}) and have non-zero first rational homology, \Cref{pr.bc} generalizes that result.
\end{remark}

Restricting attention to amenable groups $\Gamma$, \Cref{pr.bc} proves non-existence of equivariant maps 
\begin{equation*}
  C^*(\Gamma) \to \join{C^*(\Gamma)}{\Delta}{C^*(\Gamma)}
\end{equation*}
provided $H_*(\Gamma,\bQ)\ne \bQ$. We now consider a different generalization of this idea based on the existence of characters.

\begin{proposition}\label{cor.bc}
  Let $(H,\Delta)$ be a compact quantum group such that $H$ admits a character $H\to \bC$. If $K_*(H)\ne \bZ$, then there is no equivariant morphism from $H$ to $\join{H}{\Delta}{H}$.
\end{proposition}
\begin{proof}
  First, the existence of a character $H\to \bC$ as in the hypothesis entails the boundedness of the counit $\varepsilon:H\to \bC$.

  To see this, consider the convolution semigroup $S$ (assumed non-empty) of characters of $H$. $S$ embeds into the compact group $G$ of all characters of the universal version $H^u$ of $H$. Since compact sub-semigroups of compact groups are automatically subgroups (e.g. by \cite[Lemma B.1]{rs}), this proves the claim.
  
  Now let $\phi:H\to \join{H}{\Delta}{H}$ be an equivariant map. Its initial component $\phi_0$ is an equivariant map from $H$ to
\begin{equation*}
  \begin{tikzpicture}[auto,baseline=(current  bounding  box.center)]
    \path[anchor=base] (0,0) node (l) {$\Delta(H)$} +(3,0) node (r) {$H$.};
         \draw[->] (l) to[bend left=0] node[pos=.5,auto] {$\scriptstyle \Delta^{-1}$}  (r);
  \end{tikzpicture}
\end{equation*}
Any equivariant self-map $\alpha$ of $H$ is of the form $(\chi\otimes\id)\circ\Delta$ for some character $\chi:H\to \bC$, as we may set $\chi=\varepsilon\circ\alpha$. Having made the identification
\begin{equation*}
  \Delta^{-1}\circ \phi_0 = (\chi\otimes\id)\circ\Delta, 
\end{equation*}
it follows that
\begin{equation*}
  \phi_0 = (\chi\otimes\id\otimes\id)\circ \Delta^{(2)},
\end{equation*}
where
\begin{equation*}
  \Delta^{(2)} = (\Delta\otimes\id)\circ\Delta. 
\end{equation*}
But then it follows that $(\id\otimes\varepsilon)\circ\phi_0 = (\chi\otimes\id)\circ\Delta$. In particular, it is a $C^*$-isomorphism.

Composing $\phi$ with $\id\otimes\varepsilon$ thus produces a homotopy from a $C^*$-automorphism of $H$ to the character
\begin{equation*}
  (\id\otimes\varepsilon)\circ \phi_1: H\to \bC\subset H. 
\end{equation*}
This in turn gives an automorphism of $K_*(H)$ that factors through $K_*(\bC)\cong \bZ$, finishing the proof.
\end{proof}

\begin{remark}
  \Cref{cor.bc} should be compared with \cite[Theorem 3.3]{BDHrevisited}. The latter proves a stronger conclusion using a stronger hypothesis.
\end{remark}



We can dispense with any K-theoretic assumptions in \Cref{cor.bc} provided we assume the character in question is not the counit (compare with \cite[Theorem 2.2 and Corollary 2.7]{BDHrevisited} in the case $A = H$).

\begin{proposition}
  If $(H,\Delta)$ admits a non-trivial character, then there is no equivariant morphism from $H$ to $\join{H}{\Delta}{H}$.
\end{proposition}
\begin{proof}
  As in the proof of \Cref{cor.bc}, such a map would entail the existence of a path connecting an isomorphism $H\to H$ to a character. This, in turn, amounts to an analogous path for the abelianization of $H$, which is the function algebra of a classical non-trivial compact group. A contradiction now follows from the fact that non-trivial compact groups are not contractible \cite{hof}.
\end{proof}

The following proposition implies that characters in join and fusion procedures (see \cite{joinfusion}) are easily determined by the characters of the components.

\begin{proposition}\label{prop:characterjoin}
Let $B$ be a unital $C^*$-algebra, and let $X$ be a compact Hausdorff space. Fix finitely many $s_i \in X$, $1 \leq i \leq N$. For each $i$, let $B_{s_i}$ denote a unital $C^*$-subalgebra of $B$, and define $A$ as
\bes
A := \{f \in C(X, B): \text{for each } i \in \{1, \ldots, N\}, f(s_i) \in B_{s_i}\}.
\ees
Also, define $B_s := B$ for $s \in X \setminus \{s_1, \ldots, s_N\}$. Then if $\phi$ is a character on $A$, it may be decomposed as $\phi = \psi \circ \textrm{ev}_s$ for some $s \in X$ and character $\psi$ on $B_s$.
\end{proposition}
\begin{proof}
Let $J_t = \{f \in A: f(t) = 0\}$ and $I = \text{ker}(\phi)$, and note that since the $s_i$ are isolated, we have $A / J_t \cong B_t$ for each $t \in X$, implemented by evaluation at $t$. Therefore, it suffices to show that for some $t \in X$, $J_t \subseteq I$.

Just as in the commutative case, if $J_t \not\subseteq \text{ker}(\phi)$ for each $t \in X$, then $J_t + I = A$, since $I$ is of codimension one. Therefore, there exist $f_t \in I$ such that $f_t(t) = 1$. We may then cover $X$ by open sets
\bes
U_t := \{s \in X: f_t(s) \text{ is invertible in } B_s \subseteq B\}.
\ees
Any finite subcover $U_{t_1}, \ldots, U_{t_k}$ gives rise to a positive invertible function $\sum\limits_{i=1}^{k} f_{t_k}f_{t_k}^* \in I$, a contradiction of the fact that $I$ is proper.
\end{proof}

In particular, the equivariant join 
\bes
\join{A}{\delta}{H} = \{f \in C([0,1], A \otimes H): f(0) \in \delta(A), f(1) \in \bC \otimes H \}
\ees
is of the form considered in \Cref{prop:characterjoin}, so we may determine its characters. 

\begin{corollary}\label{cor:charactersiteratedjoinspaths}
Suppose there is a character $\phi$ on $\join{A}{\delta}{H}$, where $\delta$ is a free coaction of $H$ on $A$. Then the following hold.
\begin{itemize}
\item At least one of $A$ or $H$ admits a character.
\item If $H$ admits a character, then then there is a continuous path of characters on $\join{A}{\delta}{H}$ connecting $\phi$ to a character of the form $\rho \circ \text{ev}_1$, where $\rho$ is a character on $\bC \otimes H \cong H$.
\item If $A$ admits a character, then there is a continuous path of characters on $\join{A}{\delta}{H}$ connecting $\phi$ to a character of the form $\rho \circ \text{ev}_0$, where $\rho$ is a character on $\delta(A) \cong A$. 
\end{itemize}
\end{corollary}
\begin{proof}
By \Cref{prop:characterjoin}, $\phi$ is of the form $\psi \circ \text{ev}_t$, where $\psi$ is a character on $\delta(A)$ if $t = 0$, $A \otimes H$ if $0 < t < 1$, or $\bC \otimes H$ if $t = 1$. At least one of $A$ or $H$ maps into the domain of $\psi$, so at least one admits a character.

\noindent \textbf{Case I:} Suppose $H$ has a character (and hence also a counit). If $t = 1$, there is nothing left to prove. If $t \in (0, 1)$, then follow the path $\psi \circ \text{ev}_s$ for $s \in [t, 1]$. At $t = 1$, we may write the resulting character $\psi \circ \text{ev}_1$ as $\rho \circ \text{ev}_1$, for $\rho$ a character on $\bC \otimes H$, by following the embedding of $\bC \otimes H$ into $A \otimes H$. 

If $t = 0$, then $\psi$ is a character on $\delta(A)$. Because $H$ has a counit $\varepsilon$, we may compose the chain
\bes
A \otimes H \xrightarrow{\text{id}_A \otimes \varepsilon} A \xrightarrow{\delta} \delta(A) \xrightarrow{\psi} \bC
\ees
to form a character $\Psi$ on $A \otimes H$. Moreover, since
\bes
\delta \circ (\text{id}_A \otimes \varepsilon) \circ \delta = \delta \circ \text{id}_A = \delta,
\ees
it follows that $\Psi$ is actually an extension of $\psi$ from the domain $\delta(A)$ to the domain $A \otimes H$. Therefore, the original character $\phi = \psi \circ \text{ev}_0 = \Psi \circ \text{ev}_0$ may be connected via $\Psi \circ \text{ev}_s$ to a character based at $s = 1$. Similar to the above, $\Psi \circ \text{ev}_1$ may be written as $\rho \circ \text{ev}_1$ for a character $\rho$ on $\bC \otimes H$.

\noindent \textbf{Case II:} Suppose $A$ admits a character. If $t = 0$, there is nothing to prove, and if $t \in (0, 1)$, then we may again shift the character $\phi = \psi \circ \text{ev}_t$ by moving the scalar parameter to 0. If $t = 1$, then $\psi$ is a character on $\bC \otimes H$, which we may extend to a character $\Psi$ on the domain $A \otimes H$ by fixing a character $\alpha$ of $A$ and writing the composition
\bes
A \otimes H \xrightarrow{\alpha \otimes \text{id}_H} \bC \otimes H \xrightarrow{\psi} \bC.
\ees
We are then free to shift the scalar parameter of $\phi = \Psi \circ \text{ev}_1$.
\end{proof}

If $A$ is a simple quantum torus, then $A$ admits a free coaction of $C(\bZ/2\bZ)$, and $C(\bZ/2\bZ)$ certainly has characters. In this case, the join of $A$ and $C(\bZ/2\bZ)$ admits characters corresponding to those of $C(\bZ/2\bZ)$ at the $t = 1$ endpoint, but no others. When both $A$ and $H$ have characters, connectedness properties of the topological join carry over to the $C^*$-algebraic setting.

\begin{corollary}\label{cor:explicitconnectedness}
Suppose $H$ and $A$ each admit at least one character and that $\delta$ is a free coaction of $H$ on $A$. Then any character of $\join{A}{\delta}{H}$ may be connected via a path to the character $(1 \otimes \varepsilon) \circ \text{ev}_1$.
\end{corollary}
\begin{proof}
Since $H$ has a character, it also admits a counit $\varepsilon$. By \Cref{cor:charactersiteratedjoinspaths} and its proof, any character $\phi$ on $\join{A}{\delta}{H}$ may be connected to a character of the form $\psi \circ \text{ev}_0 = \Psi \circ \text{ev}_0$, where $\psi$ is a character on $\delta(A)$ which extends to a character $\Psi$ on $A \otimes H$ through the composition
\bes
A \otimes H \xrightarrow{\text{id}_A \otimes \varepsilon} A \xrightarrow{\delta} \delta(A) \xrightarrow{\psi} \bC.
\ees
The path $\Psi \circ \text{ev}_t$ connects the given character to $\Psi \circ \text{ev}_1$. However, from the boundary conditions of $\join{A}{\delta}{H}$, the range of $\text{ev}_1$ is $\bC \otimes H$, and
\bes
\Psi (1 \otimes h) = \psi \circ \delta \circ (\text{id}_A \otimes \varepsilon) (1 \otimes h) = \psi \circ \delta (\varepsilon(h)) = \psi(\varepsilon(h)) = \varepsilon(h),
\ees
so $\Psi \circ \text{ev}_1$ is the same character as $(1 \otimes \varepsilon) \circ \text{ev}_1$.
\end{proof}

Applying \Cref{cor:explicitconnectedness} to iterated joins of $H$ with itself, we conclude from the above that the character space of $E_n H$ is path connected if $n \geq 1$, analogous to the topological picture. On the other hand, consider a coaction of $H = C(\bZ/2\bZ)$ on a simple $C^*$-algebra $A$ once more. In $\join{A}{\delta}{C(\bZ/2\bZ)}$, the only two characters are of the form $\rho \circ \text{ev}_1$, and they are isolated from each other, so the assumption that $A$ also admits characters is necessary.

We conclude with results about cleftness of the iterated joins $E_n H$, when $H$ is finite-dimensional. In particular, we view $E_n H$ as a comodule algebra over the finite dimensional Hopf algebra $H$. Recall the following definition.

\begin{definition}\label{def.clft}
  Let $H$ be a Hopf algebra. An $H$-comodule algebra $A$ is called {\it cleft} if there exists a convolution-invertible $H$-comodule
  map $H\to A$ (which is called a \textit{cleaving map}).
\end{definition}

The convolution $\phi * \psi$ of linear maps $\phi, \psi: H \to A$ is given by the following commutative diagram.
\begin{equation}\label{eq.convolution}
  \begin{tikzpicture}[auto,baseline=(current  bounding  box.center)]
    \path[anchor=base] (0,0) node (a) {$H$} +(2,.7) node (b) {$H \otimes H$} +(4,.7) node (c) {$A \otimes A$}+(6,0) node (bh) {$A$};
         \draw[->] (a) to[bend left=20] node[pos=.7,auto] {$\scriptstyle \Delta$}  (b);
         \draw[->] (a) to[bend right=13] node[pos=.5,auto,swap] {$\scriptstyle \phi * \psi$} (bh);
         \draw[->] (b) to[bend left=8] node[pos=.5,auto] {$\scriptstyle \phi \otimes \psi$} (c);
         \draw[->] (c) to[bend left=20] node[pos=.1, auto] {$\scriptstyle \text{mult}$} (bh);
  \end{tikzpicture}
\end{equation}
As usual, $\Delta: H \to H \otimes H$ denotes the comultiplication, and $\text{mult}: A \otimes A \to A$ denotes the map which sends any elementary tensor $a \otimes b$ to $ab$. Further, the convolution identity is given by the formula
\bes
H \ni h \mapsto \varepsilon(h) 1_A \in A,
\ees
where $\varepsilon: H \to \bC$ is the counit of the Hopf algebra $H$. We abuse notation slightly by also calling the convolution identity $\varepsilon$. 

When $H = \bC \Gamma$ for a finite group $\Gamma$, an $H$-comodule morphism is simply a linear map $\phi: H \to A$ which respects the graded structure over $\Gamma$, and we need only be concerned with the specification of $\phi(\gamma) \in A_\gamma$. Further, since following \Cref{eq.convolution} shows that the convolution $\phi * \psi$ of two linear maps from $H$ to $A$ is given easily by the computation
\be\label{eq:cleftformulaindetail}
(\phi * \psi)\left( \sum_{\gamma \in \Gamma} a_\gamma \gamma \right) = \text{mult} \circ (\phi \otimes \psi) \circ \Delta \left(  \sum_{\gamma \in \Gamma} a_\gamma \gamma  \right) = \sum_{\gamma \in \Gamma} a_\gamma \hspace{.03 in} \phi(\gamma) \psi(\gamma),
\ee
and similarly for the reverse order, we conclude that the convolutions $\phi * \psi$ and $\psi * \phi$ are simultaneously equal to
\bes
\varepsilon: \bC \Gamma \ni \sum_{\gamma \in \Gamma} a_\gamma \gamma \mapsto \sum_{\gamma \in \Gamma} a_\gamma \in \bC
\ees
if and only if for each $\gamma \in \Gamma$, $\phi(\gamma)$ and $\psi(\gamma)$ are inverses under multiplication. It follows that $A$ is cleft over the Hopf algebra $H = \bC \Gamma$ if and only if for each $\gamma \in \Gamma$, there is an invertible element $U_\gamma = \phi(\gamma)$ in the $\gamma$-isotypic subspace of $A$.

Similarly, if $H$ is a finite-dimensional Hopf algebra, not necessarily of the form $\bC \Gamma$, consider any
finite-dimensional (right) $H$-comodule $V$. The right comodule
structure
\begin{equation*}
  \rho:V\to V\otimes H
\end{equation*}
amounts to a left module structure over the dual Hopf
algebra $H^*$, and hence we have a well-defined algebra morphism
\begin{equation*}
  \theta_\rho:H^*\to \mathrm{End}(V). 
\end{equation*}
We denote by $A_\rho$ the image in $\mathrm{End}(V)$ of this
morphism. We will allow ourselves the slight notational abuse of
substituting $V$ for the subscripts of $\theta_\rho$ and $A_\rho$.
The comodule structure is captured by an invertible element
\begin{equation*}
  u=u_V\in A_V\otimes H
\end{equation*}
defined implicitly by the equation
\begin{equation*}
  \theta_V(f) = (\mathrm{id}\otimes f)(u),\ \forall f\in H^*,
\end{equation*}
and the inverse of $u$ is its image through $\mathrm{id}\otimes S$, where
$S$ is the antipode of $H$.

\begin{proposition}\label{pr.2-join-clft}
  For any finite-dimensional complex Hopf algebra $H$, the join $E_1H = \join{H}{\Delta}{H}$
  is cleft over $H$.
\end{proposition}

\begin{proof}
  As per \Cref{def.clft}, we want to prove the existence of
  a convolution-invertible $H$-comodule morphism from $H$ to $\join{H}{\Delta}{H}$. To this end, we apply the preceding discussion to the right regular comodule
  structure of $H$ itself, i.e. $V=H$ regarded as a comodule via the
  comultiplication. In this case the subalgebra
  \begin{equation*}
    A_H\subset \mathrm{End}(H) 
  \end{equation*}
  is isomorphic to $H^*$ itself.

  Because $A_H\otimes H\cong H^*\otimes H$ is a finite-dimensional
  complex algebra, its group $GL(H^*\otimes H)$ of invertible elements
  is path connected. It follows that we can find a path
  \begin{equation*}
    p_t\in GL(H^*\otimes H),\quad p_0=u,\quad p_1=\mathrm{id};
  \end{equation*}
  it can be expressed as
  \begin{equation*}
    p_t = (\mathrm{id}\otimes \psi_t)(u)
  \end{equation*}
  for a path of linear maps
  \begin{equation*}
    \psi_t:H\to H,\quad \psi_0=\mathrm{id},\quad \psi_1=(\text{unit of }H)\circ\varepsilon.
  \end{equation*}

  Finally, define the path $\varphi_t:H\to H\otimes H$ by the following commutative diagram.

  \begin{equation}\label{eq:phi}
    \begin{tikzpicture}[auto,baseline=(current  bounding  box.center)]
      \path[anchor=base] (0,0) node (h)
      {$H$} +(2,.5) node (hhm) {$H\otimes H$} +(4,0) node (hh) {$H\otimes H$};
      \draw[->] (h) to[bend left=9] node[pos=.5,auto] {$\scriptstyle\Delta$} (hhm);
      \draw[->] (hhm) to[bend left=9] node[pos=.5,auto] {$\scriptstyle \psi_t\otimes\mathrm{id}$} (hh);
      \draw[->] (h) to[bend right=9] node[pos=.5,auto,swap] {$\scriptstyle \varphi_t$} (hh);
    \end{tikzpicture}
  \end{equation}
  It follows that the $\phi_t$ determine a map from $H$ to
  $\join{H}{\Delta}{H}$; this is the desired cleaving map. To verify this claim, note first that $\varphi=(\varphi_t)_t$
  is indeed a comodule map simply because both morphisms making up the
  upper half of \Cref{eq:phi} respect the $H$-comodule structure
  coming from the regular coaction on the right hand tensorand.

  Finally, we prove that $\varphi$ is convolution-invertible. First, note that $\psi_t:H\to H$ has a convolution inverse
  $\overline{\psi}_t$ defined simply by the requirement that
  \begin{equation*}
    p_t\text{ and }\overline{p}_t := (\mathrm{id}\otimes \overline{\psi}_t)(u)
  \end{equation*}
  be mutually inverse in $GL(H^*\otimes H)$.

  Next, we can define the convolution inverse $\overline{\varphi}_t$
  of $\varphi_t$ as the composition
  \begin{equation}\label{eq:phibar}
    \begin{tikzpicture}[auto,baseline=(current  bounding  box.center)]
      \path[anchor=base] (0,0) node (h)
      {$H$} +(2,.5) node (hhm) {$H\otimes H$} +(4,.5) node (hh2) {$H\otimes H$} +(6,0) node (hh) {$H\otimes H$};      
      \draw[->] (h) to[bend left=9] node[pos=.5,auto] {$\scriptstyle\Delta$} (hhm);
      \draw[->] (hhm) to[bend left=6] node[pos=.5,auto] {$\scriptstyle\tau$} (hh2);      
      \draw[->] (hh2) to[bend left=9] node[pos=.5,auto] {$\scriptstyle \overline{\psi}_t\otimes S$} (hh);
      \draw[->] (h) to[bend right=12] node[pos=.5,auto,swap] {$\scriptstyle \overline{\varphi}_t$} (hh);
    \end{tikzpicture}
  \end{equation}
  where $\tau$ denotes the flip.
\end{proof}
\begin{remark}
It is crucial to note that in the above, there is no claim that the cleaving map respects the adjoint operation. In some cases, compatibility with the adjoint operation is impossible, as discussed in \Cref{cor.cleftad}.
\end{remark}

When $H = \bC \Gamma$ is the group algebra of a finite group $\Gamma$, cleftness properties of any join $E_n H$ may be described explicitly.
\begin{theorem}\label{thm:cleftsummary}
Fix $n \geq 2$, let $\Gamma$ be a nontrivial finite group, and define the complex Hopf algebra $H = \bC\Gamma$. Then $E_n H$ is cleft over $H$ if and only if $\Gamma$ is perfect.
\end{theorem}
\begin{proof}
Recall from \Cref{eq:cleftformulaindetail} that if $H = \bC \Gamma$, an $H$-comodule algebra $A$ is cleft over the Hopf algebra $H$ if and only if for each $\gamma \in \Gamma$, there is an invertible element $U_\gamma \in A_\gamma$. If $\Gamma$ is perfect, then we may select invertibles $U_\gamma$ in each $(E_n H)_\gamma$ by \Cref{thm:perfectunitaries}. It follows that $E_n H$ is cleft over $H$. 

On the other hand, if $\Gamma$ is not perfect, there is a quotient $H \to K$, where $K = \bC \Lambda$ for a nontrivial cyclic group $\Lambda = \bZ/k\bZ$. If $E_n H$ is cleft over the Hopf algebra $H$, then there is an invertible element in $(E_n H)_\gamma$ for each $\gamma \in \Gamma$, and applying a pointwise quotient shows there is also an invertible element in $(E_n K)_\lambda$ for each $\lambda \in \Lambda$. Now, $E_n K$ is equivariantly isomorphic to the $C^*$-algebra of continuous complex-valued functions on the $(n-1)$-connected space $X = E_n(\bZ/k\bZ) = (\bZ/k\bZ)^{*n+1}$, where increasing connectivity of the iterated joins may be derived from \cite[Proposition 4.4.3]{matousek}. An invertible element in the generating isotypic subspace of $E_n K$ may be rescaled to produce an equivariant map $X \to \bS^1$, where $X$ is acted upon freely and diagonally by $\bZ/k\bZ$, and $\bS^1$ is acted upon freely by $\bZ/k\bZ$ through an order $k$ rotation. This contradicts the remark on page 68 of \cite{dold}, as the connectivity of $X$ is at least as big as the dimension of $\bS^1$.
\end{proof}

\Cref{thm:cleftsummary} implies that for \textit{finite} groups, \Cref{thm:perfectunitaries} is optimal for $n \geq 2$, in that the only nontrivial finite groups $\Gamma$ for which the theorem holds are perfect. We close by noting that in many cases, it is impossible to arrange for a cleaving map $H \to E_n H$ which also respects the adjoint operation. Specifically, when $H = \mathbb{C} \Gamma$, this property can be completely characterized.

\begin{corollary}\label{cor.cleftad}
Suppose $H = \bC \Gamma$ where $\Gamma$ is a nontrivial group. Then a cleaving map $H \to E_1 H$ which respects the adjoint operation exists if and only if $|\Gamma|$ is odd. Moreover, if $n \geq 2$ and $E_n H$ is cleft over $H$, then no cleaving map $H \to E_n H$ respects the adjoint operation.
\end{corollary}
\begin{proof}
A cleaving map $H \to E_n H$ is determined by a selection of invertible elements $U_\gamma$ in the $\gamma$-isotypic subspace of $E_n H$. If the cleaving map respects the adjoint, then this just means $U_{\gamma^{-1}} = U_\gamma^*$.

Suppose $n = 1$ and $|\Gamma|$ is odd. Then the only element $\gamma \in \Gamma$ with $\gamma^{-1} = \gamma$ is the identity, so we may partition $\Gamma \setminus \{e\} = X_1 \cup X_2$ by splitting inverse pairs. By defining $U_\gamma \in (E_n H)_\gamma$ for $\gamma \in X_1$ first, we may ensure that $U_{\gamma^{-1}} = U_\gamma^*$ holds by design.

On the other hand, if $n = 1$ and $|\Gamma|$ is even, then there exists an order two element $\gamma \in \Gamma$, so the condition $U_{\gamma^{-1}} = U_\gamma^*$ implies that $U_\gamma$ is self-adjoint. Since $U_\gamma$ is also an invertible element of $(E_1 H)_\gamma$, application of the counit pointwise on the right tensorand produces a path of self-adjoint invertibles in $H$. This path of self-adjoint invertibles connects an element $h \in H$ with $\sigma(h) = -\sigma(h)$ to a scalar, which is a contradiction.

Similarly, if $n \geq 2$, then a cleaving map $H \to E_n H$ only exists when $\Gamma$ is perfect, which implies $|\Gamma|$ is even. Since any self-adjoint invertible in $(E_n H)_\gamma$ produces a self-adjoint invertible in $(E_1 H)_\gamma$ through a quotient, a similar argument as above implies that no cleaving map will respect the adjoint.
\end{proof}


\section*{Acknowledgments}

We are grateful to Piotr M. Hajac and Mariusz Tobolski for helpful comments. The anonymous referees' suggestions were also of great help in improving the draft.

\bibliography{invariantsrefs}{}
\bibliographystyle{plain}
\addcontentsline{toc}{section}{References}

\end{document}